\documentclass[a4paper,10pt]{article}

\usepackage{amsthm}
\usepackage{amsmath}
\usepackage{amssymb}
\usepackage{amsfonts}
\usepackage{fancybox}
\usepackage[all]{xy}

\theoremstyle{definition}
\newtheorem{df}{Definition}[section]
\newtheorem{thm}[df]{Theorem}
\newtheorem{lem}[df]{Lemma}
\newtheorem{cor}[df]{Corollary}
\newtheorem{prop}[df]{Proposition}

\newtheorem{rem}[df]{Remark}
\newtheorem{eg}[df]{Example}

\newcommand\ph{\varphi}

\newcommand\ra{\rightarrow}

\DeclareMathOperator{\Ord}{Ord}
\newcommand\Set{{\textbf{{\rm Set}}}}
\newcommand\iso{{\simeq}}
\DeclareMathOperator{\id}{id}
\newcommand\sm{\setminus}
\newcommand\mc{\mathclose}

\newcommand\N{\mathbb{N}}

\newcommand\V{\mathcal{V}}
\DeclareMathOperator{\Clo}{Clo}
\DeclareMathOperator{\Inv}{Inv}

\DeclareMathOperator{\Con}{Con}
\DeclareMathOperator{\Sub}{Sub}
\DeclareMathOperator{\ari}{ar}
\DeclareMathOperator{\Ht}{Ht}
\DeclareMathOperator{\Cat}{Cat}

\newcommand\K{\mathcal{K}}

\newcommand\dom{{\rm dom}}
\newcommand\cod{{\rm cod}}

\title{Matrix product is many-sorted algebra}
\author{Shohei Izawa\footnote{sa9m02@math.tohoku.ac.jp}}
\date{}

\begin{document}
\maketitle
\begin{abstract}
It is known that a category of many-sorted algebras on pure sets of similarity type is ``concretely equivalent" to a category of single-sorted algebras.
In this paper, we characterize a single-sorted variety that corresponds to a many-sorted variety. Such variety is also characterized by the condition that is decomposable with respect to matrix product.
\end{abstract}

\section{Introduction} \label{s-introduction}

Recent paper \cite{MRS} proves that a category of many-sorted algebras on
pure sets (Definition \ref{pure-set} below) is categorically equivalent to
a category of single-sorted algebras.
In this paper, we consider on the case of varieties.
We characterize a single-sorted variety that corresponds to a many-sorted variety.
Moreover, we present one to one correspondence between many-sorted varieties
and single-sorted varieties ``with a diagonal pair".
We also exhibit this correspondence preserves underlying sets in the sense
that is compatible with the functor described in Proposition \ref{pureset-and-sets}.

On the other hand, there is a notion, studied in clone theory and study on category of algebras, called matrix product of algebras (e.g.\cite{Iza},\cite{Kea}).
Matrix product is a construction of a new algebra and the constructed algebra always has a diagonal pair.
We also prove that an algebra (or a variety) is decomposable with respect to the notion of matrix product if and only if the algebra (or the variety) has a diagonal pair.

\section{Preliminaries} \label{s-preliminaries}

\subsection{Many-sorted variety and clone} \label{s-many-sorted-variety-and-clone}

As single-sorted case, there is a natural correspondence between 
many-sorted varieties and algebras of terms; clones of many-sorted varieties.

In this section, we quickly introduce the concept of many-sorted variety and
many-sorted clone, then we explain the correspondence between varieties and clones.

Through this paper, we explain an algebra that has (possibly) infinitely
many-sorts corresponds to an infinitary single-sorted algebra.
Therefore, we consider (possibly) infinitary single-sorted/many-sorted
algebras in this paper.
However, almost all results of this paper seem new even if
we consider only on the case that all (classes of) algebras are finitary
and finitely many-sorted.
If the reader is interested only in finitary, finitely
many-sorted algebras, please read the arity bound $\kappa$ in the paragraph
the countable cardinal $\aleph_0$ and the number of sorts $S$ a finite cardinal.

In this paper, we use $\subset$ for the subset relation includes equality,
namely, $A\subset B$ means $x\in A\Rightarrow x\in B$ holds for all $x$.
We write $A\subsetneq B$ the condition $A\subset B$ and $A\neq B$. 

First, we describe the usual type-based definition of algebras and varieties.

\begin{df}
Let $S$ be a cardinal.
\begin{enumerate}
\item
An $S$-sorted type is a tuple $(F,\ari,\dom,\cod)$ that satisfies the following conditions:
\begin{itemize}
\item
$F$ is a set.
\item
$\ari:F\ra {\rm Card}$, where ${\rm Card}$ is the class of all cardinals.
\item
$\dom,\cod$ are functions defined on $F$ and satisfy
$\dom(f):\ari(f)\ra S$, $\cod(f)\in S$
for each $f\in F$.
\end{itemize}
\item
For an $S$-sorted type $F$, $F$-algebra is a pair $(A,\tau)$
that satisfies the following conditions:
\begin{itemize}
\item
$A=(A_s)_{s\in S}$ is an $S$-indexed family of sets.
\item
$\tau$ is a function defined on $F$ and
$\tau(f)$ is a mapping $\prod_{i\in\ari(f)}A_{\dom(f)(i)}\ra A_{\cod(f)}$
for each $f\in F$.
\end{itemize}
\item
For an infinite cardinal $\kappa$, $F$ is said $<\mc\kappa$-ary if
$\ari(f)<\kappa$ hold for all $f\in F$.
\item
The height of $F$ is defined as
\[
\Ht(F):=
\begin{cases}
\omega_0& \text{if }F\text{ is finitary},\\
\min\{\kappa\in\Ord\mid {\rm cf}(\kappa) >\ari(f)\text{ for all }f\in F\}& \text{ if }F\text{ is infinitary},
\end{cases}
\]
where $\omega_0$ is the minimum infinite ordinal,
$\Ord$ is the class of all ordinals, and ${\rm cf}(\kappa)$ is
the cofinality of $\kappa$, i.e., the minimum cardinal $\alpha$
such that there exists an $\alpha$-indexed family $(\lambda_i)_{i\in\alpha}$
of cardinals that satisfies $\lambda_i<\kappa$ (for all $i\in\alpha$) and
$\sup_{i\in\alpha} \lambda_i=\kappa$.
\item
For an $S$-sorted type $F$, an ordinal $\alpha$, a cardinal $\lambda$,
a mapping $v:\lambda\ra S$ and an element $s\in S$,
the set $T_{\alpha,(\lambda,v,s)}$ of all $(\lambda,v)$-ary $s$-valued terms of $F$
with complexity less than $\alpha$
is inductively defined as follows:
\begin{itemize}
\item
$T_{0,(\lambda,v,s)}:=\{x_i\mid v(i)=s,i\in\lambda\}$.
(The set of variable symbols.)
\item
$T_{\alpha+1,(\lambda,v,s)}:=T_{\alpha,(\lambda,v,s)}\cup
\{(f,(t_i)_{i\in \ari(F)})\mid 
 f\in F,t_i\in T_{\alpha,(\lambda,v,\dom(f)(i))},\cod(f)=s\}$.
\item
$T_{\alpha,(\lambda,v,s)}:=\bigcup_{\beta<\alpha} T_{\beta,(\lambda,v,s)}$
if $\alpha$ is a limit ordinal.
\end{itemize}
The set of all $(\lambda,v)$-ary $s$-valued terms of $F$ is $T_{\Ht(F),(\lambda,v,s)}$.
\item
The action of terms to an algebra $(A,\tau)$ is inductively defined as follows:
\begin{itemize}
\item
For $x_i\in T_{0,(\lambda,v,v(i))}$, $\tau(x_i):\prod_{j\in \lambda}A_{v(j)}\ra A_{v(i)}$ is defined as $(a_j)_{j\in\lambda}\mapsto a_i$.
\item
If $t=(f,(t_i)_{i\in \ari(F)})$,
$\tau(t):\prod_{j\in \lambda}A_{v(j)}\ra A_{\cod(f)}$ is defined as
\[
(a_j)_{j\in\lambda}\mapsto \tau(f)(\tau(t_i)(a_j)_{j\in\lambda})_{i\in\ari(F)}.
\]
\end{itemize}
\item
For $s\in S$, An $s$-sorted identity is a pair of two $s$-valued terms.
The tuple $(t_1,t_2)$ is usually denoted by
$t_1=t_2$ when it is considered as an identity.
A relation $(A,\tau) \models t_1=t_2$ between an algebra $(A,\tau)$ and 
an equation $t_1=t_2$ is defined by $\tau(t_1)=\tau(t_2)$.
\item
An equational theory is a set of identities.
An algebra $A$ and an equational theory $E$,
$A\models E$ means $A\models e$ for all $e\in E$.
A pair $(F,E)$ of type $F$ and equational theory $E$ is called a 
type with equational theory.
\item
A class $\V$ of $F$-algebras is said a variety if there exists
an equational theory $E$ such that \mbox{$\V=\{A\mid A\models E\}$}.
The variety of $F$-algebras defined by $E$ is denoted by $\V_{(F,E)}$.
\end{enumerate}
\end{df}

The above is traditional style of the definition of algebras.
On the other hand, as the single-sorted finitary case,
we can define essentially the same notion of algebras by
the following clone-based description.
(At least in the author's opinion,) clone-based definition is simpler
than type-based definition.
By this reason, we use the clone-based description through this paper.
Note that, the clone-based description essentially includes type-based description
as the case that the clone is freely generated.

\begin{df}
Let $S$ and $\kappa$ be cardinals.
An $S$-sorted $<\mc\kappa$-ary clone $M$ is a many-sorted algebra
satisfying the following conditions:
\begin{itemize}
\item
(Sort) The set of all sorts of $M$ consists of tuples $(\lambda,v,s)$,
where $\lambda<\kappa$ is a cardinal, $v:\lambda\ra S$ and $s\in S$.
\item
(Operation) $M$ has the following two types of operations:
\begin{itemize}
\item
(Projection) For $\lambda<\kappa$, $v:\lambda\ra S$ and $i\in \lambda$,
$M$ has a nullary operation $\pi_{(\lambda,v,i)}\in M_{(\lambda,v,v(i))}$.
\item
(Composition) For $\lambda_k<\kappa, v_k:\lambda_k\ra S$ ($k=1,2$) and $s\in S$,
$M$ has an operation 
\[
c_{(\lambda_1,v_1,s),(\lambda_2,v_2)}:
  M_{(\lambda_1,v_1,s)}\times \prod_{i\in\lambda_1}M_{(\lambda_2,v_2,v_1(i))}
  \ra M_{(\lambda_2,v_2,s)}
\]
\end{itemize}
\item
(Axiom) $M$ satisfying the following equations:
\begin{itemize}
\item
(Associativity) For $\lambda_k<\kappa,v_k:\lambda_k\ra S$ ($k=1,2,3$) and $s\in S$,
\begin{align*}
&c_{(\lambda_1,v_1,s),(\lambda_3,v_3)}(x,
  (c_{(\lambda_2,v_2,v_1(i)),(\lambda_3,v_3)}(y_i,(z_{j})_{j\in\lambda_2}))_{i\in\lambda_1})\\
=&
c_{(\lambda_2,v_2,s),(\lambda_3,v_3)}
(c_{(\lambda_1,v_1,s),(\lambda_2,v_2)}(x,(y_i)_{i\in\lambda_1}),
  (z_j)_{j\in\lambda_2}).
\end{align*}
\item
(Outer identity law) For $\lambda_k<\kappa,v_k:\lambda_k\ra S$ ($k=1,2$)
and $i_0\in\lambda_1$
\[
c_{(\lambda_1,v_1,v_1(i_0)),(\lambda_2,v_2)}
(\pi_{(\lambda_1,v_1,i_0)},(x_i)_{i\in\lambda_1})
=x_{i_0}.
\]
\item
(Inner identity law) For $\lambda<\kappa,v:\lambda\ra S$ and $s\in S$, 
\[
c_{(\lambda,v,s),(\lambda,v)}(x,(\pi_{(\lambda,v,i)})_{i\in\lambda})
=x.
\]
\end{itemize}
\end{itemize}
\end{df}
If there are no possibility of confusion, we omit the subscript of $M$ or $c$.

\begin{df}
Let $S$ and $\kappa$ be a cardinals, $M$ be an $S$-sorted $<\mc\kappa$-ary clone.
A tuple $(A,\tau)$ is said an $M$-algebra if the following conditions hold:
\begin{itemize}
\item
$A=(A_s)_{s\in S}$ is an $S$-indexed family of sets.
\item
$\tau=(\tau_{(\lambda,v,s)})$ is a family of mappings, where $(\lambda,v,s)$ runs
all tuples that satisfy $\lambda<\kappa$, $v:\lambda \ra S$ and $s\in S$,
and $\tau_{(\lambda,v,s)}$ is defined on $M_{(\lambda,v,s)}$ and
$\tau_{(\lambda,v,s)}(f)$ is a mapping $\prod_{i\in\lambda}A_{v(i)}\ra A_s$
for each $f\in M_{(\lambda,v,s)}$.
\item
For $f\in M_{(\lambda_1,v_1,s)}$,
$(g_i)_{i\in\lambda_1}\in \prod_{i\in\lambda_1}M_{(\lambda_2,v_2,v_1(i))}$
and $(a_j)_{j\in\lambda_2}\in\prod_{j\in\lambda_2}A_{v_2(j)}$,
the following equation holds:
\[
\tau_{(\lambda_2,v_2,s)}(c_{(\lambda_1,v_1,s),(\lambda_2,v_2)}(f,(g_i)_{i\in\lambda_1}))(a_{j})_{j\in\lambda_2}
=\tau_{(\lambda_1,v_1,s)}(f)(\tau_{(\lambda_2,v_2,v_1(i))}(g_i)(a_j)_{j\in\lambda_2})_{i\in\lambda_1}.
\]
\end{itemize}
The class of all $M$-algebras is denoted by $\V(M)$.
The category of $M$-algebras, that is, the class of objects is $\V(M)$
and the set of morphisms $A$ to $B$ is the set of all homomorphisms $A\ra B$,
is denoted by $\Cat(\V(M))$.
\end{df}

In this paper, all classes of algebras appear in the text consists of algebras of a common clone.

The case $S$ is a singleton, $S$-sorted clone is simply called clone,
or single-sorted clone. In this case, the sort $(\lambda,v,s)$ is
simply denoted by $\lambda$, $\lambda$-ary $i$-th projection is
denoted by $\pi_{(\lambda,i)}$.

Next, we quickly explain connection between type-based
definition and clone-based definition.

\begin{df}
Let $S$ be a cardinal, $(F,E)$ be an $S$-sorted type with equational theory.
Let $\kappa>S$ be an infinite cardinal.
We define the $S$-sorted $<\mc\kappa$-ary clone of $F$ modulo $E$ as follows:
\begin{itemize}
\item
The underlying set of a sort $(\lambda,v,s)$
is $T_{(\lambda,v,s)}/\mc\sim_E$ where
$T_{(\lambda,v,s)}$ is the set of all $(\lambda,v)$-ary $s$-valued terms of $F$.
The equivalence $t_1\sim_E t_2$
is defined by the condition
\[
``A\models E\Rightarrow A\models t_1=t_2\text{ for all }
F\text{-algebra }A".
\]
\item
Fundamental operations are defined as
\begin{align*}
\pi_i:=&x_i/\mc \sim_E, \\
c(t/\mc\sim_E,(u_i/\mc\sim_E)_{i\in\lambda})
:=&(t\circ (u_i)_{i\in \lambda})/\mc\sim_E.
\end{align*}
This clone is denoted by ${\cal M}_{\kappa}(F,E)$, or 
${\cal M}_\kappa(\V_{(F,E)})$ by using the variety $\V_{(F,E)}$
defined by the equational theory $E$.
\end{itemize}
\end{df}

In this paper, we consider only on varieties.
By this reason we can define the notion of definitional equivalence
by a term of the corresponding clone.
\begin{df}
Let $S$ be a cardinal, $\kappa$ be an infinite cardinal and
$\V_1,\V_2$ be $S$-sorted $<\mc\kappa$-ary varieties.
The varieties $\V_1$ and $\V_2$ are said to be definitionally equivalent to each other
if ${\cal M}_{\kappa}(\V_1)$ is isomorphic to ${\cal M}_\kappa(\V_2)$.
\end{df}

\begin{prop}
Let $S$ be a non-zero cardinal, $\kappa>S$ be an infinite cardinal.
\begin{enumerate}
\item
If $\V$ is an $S$-sorted $<\mc\kappa$-ary variety, then
the clone ${\cal M}_\kappa(\V)$ of terms of $\V$ is
an $S$-sorted $<\mc\kappa$-ary clone.
\item
For an $S$-sorted $<\mc\kappa$-ary clone $M$,
there exists an $S$-sorted $<\mc\kappa$-ary variety $\V$
such that ${\cal M}_{\kappa}(\V)$ is isomorphic to $M$.
\item
Let $\kappa'>\kappa$ be a cardinal,
$\V_1,\V_2$ be $S$-sorted $<\mc\kappa$-ary varieties.
Then ${\cal M}_{\kappa'}(\V_1)\iso{\cal M}_{\kappa'}(\V_2)$ if and only if
${\cal M}_{\kappa}(\V_1)\iso{\cal M}_{\kappa}(\V_2)$.
\end{enumerate}
\end{prop}

\begin{rem}
For simplification, we use, for example, the following variable notations through this paper:
\begin{itemize}
\item
In the case $\lambda$ is finite, the sort $(\lambda,v,s)$ of many-sorted clone is denoted by $(v(0),\dots,v(\lambda-1))\ra s$.
The case $\lambda=1$, $(1,v,s)$ is denoted by $v(0)\ra s$.
The case $v$ is constant mapping, $v(i)=t$ for all $i$, $(n,v,s)$ is denoted by $nt\ra s$.
\item
The composition operation $c$ of a clone is sometimes 
written as a composition of mappings.
For example, $c(x,(y_i)_{i\in\lambda})$ is denoted by
$x\circ(y_i)_{i\in\lambda}$ or $x(y_i)_{i\in \lambda}$.
\item
we write $(b,(a_s)_{s\in S\sm\{s_0\}})$ the tuple $(x_s)_{s\in S}$ such that
$x_{s_0}=b$ and $x_s=a_s$ for $s\in S$.
The notation $((a_s)_{s\in T},(b_s)_{s\in S\sm T})$ means the tuple $(x_s)_{s\in S}$
such that $x_s=a_s$ if $s\in T$ and $x_s=b_s$ if $s\in S\sm T$.
\item
If we write $\pi_{(X,v,x)}\in M_{(X,v,s)}$,
where $M$ is an $S$-sorted clone, $X$ is a set, $v:X\ra S$, $s\in S$ and $x\in X$,
we implicitly fix a bijection $\ph:X\ra |X|$ and
$M_{(X,v,s)}$ is identified with $M_{(|X|,v\circ\ph^{-1},s)}$ and
$\pi_{(X,v,x)}$ means
$\pi_{(|X|,v\circ \ph^{-1},\ph(x))}$.
\end{itemize}
\end{rem}

\subsection{Pure set and diagonal algebra} \label{s-pure-set-and-diagonal-algebra}

Let $\V$ be a class of many-sorted algebras. If there is a class of
single-sorted algebras naturally corresponding to $\V$, then all algebras
in $\V$ have the property that
the underlying sets of each sorts of the algebra are simultaneously empty
or simultaneously non-empty.
Here, ``naturally corresponding to $\V$" precisely means there is a natural equivalence between $\V$ and a single-sorted variety
that is compatible with the natural equivalence described in Proposition \ref{pureset-and-sets}.

Through this paper, $S$ be a fixed non-zero cardinal and $\kappa$ be a fixed infinite cardinal
that $\kappa>S$.

\begin{df}[cf.\ {\cite[Definition 2.1]{MRS}}]\label{pure-set}
\mbox{}
\begin{enumerate}
\item
A family $(A_s)_{s\in S}$ of sets said to be pure
if one of the following conditions holds:
One condition is $A_s=\emptyset$ for all $s\in S$.
The other is $A_s\neq \emptyset$ for all $s\in S$.
\item
Let $A=(A_s)_{s\in S}$ and $B=(B_s)_{s\in S}$ be pure sets.
A morphism $f:A\ra B$ of pure sets is a tuple $f=(f_s)_{s\in S}$ such that
$f_s:A_s\ra B_s$ for all $s\in S$.
\item
A category $\Set^S$ of $S$-sorted pure sets is the category that the class of objects consists of all $S$-sorted pure sets
and the set of morphisms from a pure set $A$ to a pure set $B$ is
the set of all morphisms of pure sets $A\ra B$.
\end{enumerate}
\end{df}

\begin{df}[cf.\ {\cite[Definition 3.1]{MRS}}]
The following single-sorted $<\mc\kappa$-ary clone $C_{{\cal D}_S}$
is called the clone corresponds
to the variety of diagonal algebras of degree $S$:
\begin{itemize}
\item (Generators)
$C_{{\cal D}_S}$ is generated by an element $d\in (C_{{\cal D}_S})_S$ of sort $S$.
(Namely, $d$ is an $S$-ary term.)
\item (Relations)
The set of fundamental relations consists of the following only one equation,
called diagonal identity:
\[
d(d(\pi_{(S\times S,(s,t))})_{t\in S})_{s\in S}=d(\pi_{(S\times S,(s,s))})_{s\in S}.
\]
\end{itemize}
A $C_{{\cal D}_S}$-algebra is called a diagonal algebra of degree $S$.
The variety of diagonal algebras of degree $S$ is denoted by ${\cal D}_S$.
\end{df}

\begin{prop}[{\cite[Theorem 3.4]{MRS}}]\label{pureset-and-sets}
The category $\Set^S$ is categorically equivalent to the category ${\cal D}_S$
via the following categorical equivalence $\ph$:
\begin{itemize}
\item
(Objects correspondence) For $A=(A_s)_{s\in S}\in \Set^S$, 
$\ph(A)$ is the following diagonal algebra:
\begin{itemize}
\item
The underlying set is the product set $\prod_{s\in S}A_s$.
\item
For $(x_{s,t})_{t\in S}\in \prod_{t\in S}A_t$ ($s\in S$), the value
of $d$ is defined as
\[
d((x_{s,t})_{t\in S})_{s\in S}:=(x_{s,s})_{s\in S}.
\]
\end{itemize}
\item
(Morphism correspondence)
For a morphism $f=(f_s)_{s\in S}:(A_s)_{s\in S}\ra (B_s)_{s\in S}$ of pure set,
the corresponding homomorphism $\ph(f):\ph((A_s)_{s\in S})\ra \ph((B_s)_{s\in S})$
is defined by $(x_s)_{s\in S}\mapsto (f_s(x_s))_{s\in S}$.
\end{itemize}
\end{prop}

\section{Homogenization of many-sorted variety} \label{s-homogenization-of-many-sorted-variety}

Let $M$ be an $S$-sorted clone and consider the following condition:
The variety of all $M$-algebras is concretely equivalent to a
single-sorted variety, that is, categorically equivalent
via a categorical equivalence that compatible with the categorical equivalence
described in Proposition \ref{pureset-and-sets}.

This condition contains the condition that ``the underlying family of sets of
an $M$-algebra is pure".
The quoted condition is equivalent to that $M$ satisfies the next definition.

\begin{df}\label{pure-clone}
An $S$-sorted clone $M$ is said to be pure if 
$M_{s_1\ra s_2}\neq \emptyset$ for all $s_1,s_2\in S$.
\end{df}

In this section, we prove this condition also be sufficient
that there exists a corresponding single-sorted variety.

Next, we explain the construction of the corresponding
single-sorted clone and variety from a many-sorted (not necessarily pure) clone.

\begin{df}\label{homog-clone}
Let $M$ be an $S$-sorted $<\mc\kappa$-ary clone.
A single-sorted $<\mc\kappa$-ary clone ${\rm H}(M)$ defined as follows is
said a homogenization of $M$:
\begin{itemize}
\item (Underlying set)
For a cardinal $\lambda<\kappa$, the underlying set ${\rm H}_\lambda(M)$ of a sort $\lambda$
(the sort corresponding to all $\lambda$-ary terms) consists of all tuple
$(f_s)_{s\in S}$, where $f_s\in M_{(\lambda\times S,p_2,s)}$.
Here, and through this paper, the symbol $p_2$ is used for the second projection
of a product set.
(In this case, $p_2$ is the mapping $(i,s')\mapsto s'$ from $\lambda\times S$ to $S$.)
\item (Projection) The nullary operation $\pi_{(\lambda,i)}\in {\rm H}_\lambda(M)$ corresponds to
$\lambda$-ary $i$-th projection is define as
\[
\pi_{(\lambda,i)}:=(\pi_{(\lambda\times S,p_2,(i,s))})_{s\in S}.
\]
\item (Composition)
For $f=(f_s)_{s\in S}\in {\rm H}_{\lambda_1}(M)$ and
$g_i=(g_{i,s})_{s\in S}\in {\rm H}_{\lambda_2}(M)$ ($i\in \lambda_1$)
the composition is defined as
\[
c(f,(g_i)_{i\in\lambda_1}):=(c(f_s,(g_{i,t})_{(i,t)\in\lambda_1\times S}))_{s\in S}.
\]
\end{itemize}
\end{df}

\begin{df}\label{homog-alg}
Let $M$ be an $S$-sorted $<\mc\kappa$-ary clone.
Let $A=(A_s)_{s\in S}$ be an $M$-algebra.
We define the homogenization ${\rm H}(A)$ of $A$ as the following ${\rm H}(M)$-algebra:
\begin{itemize}
\item
The underlying set is $\prod_{s\in S}A_s$.
\item
The action of $f=(f_s)_{s\in S}\in {\rm H}_\lambda(M)$ is
\[
((a_{i,t})_{t\in S})_{i\in\lambda}\mapsto (f_s(a_{i,t})_{(i,t)\in\lambda\times S})_{s\in S}.
\]
\end{itemize}
\end{df}

\begin{prop}
Let $M$ and $A$ be as Definition \ref{homog-alg}.
Then ${\rm H}(A)$ is actually ${\rm H}(M)$-algebra. Namely,
the compatibility condition 
\[
f(g_i(a_j)_{j\in\lambda_2})_{i\in\lambda_1}
=c(f,(g_i)_{i\in \lambda_1})(a_j)_{j\in\lambda_2}
\]
holds for $\lambda_1,\lambda_2<\kappa$, $f\in {\rm H}_{\lambda_1}(M)$,
$g_i\in{\rm H}_{\lambda_2}(M)$ ($i\in \lambda_1$) and
$a_j\in {\rm H}(A)$ ($j\in \lambda_2$).
\end{prop}
\begin{proof}
Let $f=(f_s)_{s\in S}$, $g_i=(g_{i,s})_{s\in S}$ and $a_j=(a_{j,s})_{s\in S}$.
Then
\begin{align*}
f(g_i(a_j)_{j\in\lambda_2})_{i\in\lambda_1}
&=f((g_{i,t}(a_{j,u})_{(j,u)\in\lambda_2\times S})_{t\in S})_{i\in\lambda_1}\\
&=(f_s(g_{i,t}(a_{j,u})_{(j,u)\in\lambda_2\times S})
  _{(i,t)\in\lambda_1\times S})_{s\in S}\\
&=(c(f_s,(g_{i,t})_{(i,t)\in\lambda_1\times S})
  (a_{j,u})_{(j,u)\in\lambda_2\times S})_{s\in S}\\
&=c(f,(g_i)_{i\in \lambda_1})(a_j)_{j\in\lambda_2}.
\end{align*}
\end{proof}

Homogenization is extended as a functor from a category of 
many-sorted algebras to the category of corresponded single-sorted
algebras.
\begin{prop}\label{many-single}
Let $M$ be an $S$-sorted $<\mc\kappa$-ary clone, $A,B$ be $M$-algebras,
and $\ph:A\ra B$ be a homomorphism of $M$-algebras.
Then ${\rm H}(\ph):(a_s)_{s\in S}\mapsto (\ph_s(a_s))_{s\in S}$ is a
homomorphism ${\rm H}(A)\ra {\rm H}(B)$ of ${\rm H}(M)$-algebras.
Furthermore, the following correspondence ${\rm H}$ is a functor from the
category $\Cat(\V(M))$ of $M$-algebras to the category $\Cat(\V({\rm H}(M)))$
of ${\rm H}(M)$-algebras:
\begin{itemize}
\item
(Object correspondence) $A\mapsto {\rm H}(A)$.
\item
(Morphism correspondence) $\ph\mapsto {\rm H}(\ph)$.
\end{itemize}
\end{prop}
\begin{proof}
Let $f=(f_s)_{s\in S}\in {\rm H}_\lambda(M)$ and $a_i=(a_{i,s})_{s\in S}$ for $i\in\lambda$. Then
\begin{align*}
{\rm H}(\ph)(f(a_i)_{i\in\lambda})
&={\rm H}(\ph)(f_s(a_{i,t})_{(i,t)\in\lambda\times S})_{s\in S}\\
&=(\ph_s f_s(a_{i,t})_{(i,t)\in\lambda\times S})_{s\in S}\\
&=(f_s(\ph_t(a_{i,t}))_{(i,t)\in\lambda\times S})_{s\in S}\\
&=f({\rm H}(\ph)(a_i))_{i\in\lambda}.
\end{align*}
The correspondence ${\rm H}$ being a functor means
\[
{\rm H}(\id_A)=\id_{{\rm H}(A)}
\]
holds for all $A\in\V(M)$, and
\[
{\rm H}(\psi)\circ{\rm H}(\ph)={\rm H}(\psi\circ\ph)
\]
hold for all homomorphisms $\psi,\ph$ that $\cod(\ph)=\dom(\psi)$.
It is easily verified.
\end{proof}

Homogenization of a pure clone always has terms that satisfy the following identities.
In section \ref{s-heterogenization}, we prove the converse, i.e.,
single-sorted clone that has terms satisfy these identities
is isomorphic to homogenization of a many-sorted pure clone.

\begin{prop}
Let $M$ be an $S$-sorted $<\mc\kappa$-ary pure clone.
Assume $d\in {\rm H}_S(M)$ and $e_s\in {\rm H}_1(M)$ ($s\in S$)
be terms satisfying the following conditions:
\begin{itemize}
\item
$d$ is defined as $d=(\pi_{(S\times S,p_2,(s,s))})_{s\in S}$.
(Intuitively, the term operation of $d$ acts as
\[
A^S\ni ((a_{s,t})_{t\in S})_{s\in S}\mapsto (a_{s,s})_{s\in S}\in A
\]
for each ${\rm H}(M)$-algebra $A$.)
\item
Each term $e_s$ satisfies follows:
\begin{itemize}
\item
$e_s\circ(\pi_{(S,\id_S,s)},(x_u)_{u\in S\sm\{s\}})=e_s$
(Intuitively, the value of term operation $e_s(a_u)_{u\in S}$
depends only on $s$-component $a_s$.)
\item
If $e_s=(e_{s,t})_{t\in S}$, where $e_{s,t}\in M_{(S,\id_S,t)}$,
then $e_{s,s}=\pi_{(S,\id_S,s)}$.
(Intuitively, the value $e_{s,s}(a_u)_{u\in S}$ of $s$-component
is $a_s$.)
\end{itemize}
\end{itemize}
Then the following equations hold in ${\rm H}(M)$.
\begin{enumerate}
\item
$e_s\circ d=e_s\circ \pi_{(S,s)}$.
\item
$d\circ(e_s\circ \pi_{(S,s)})_{s\in S}=d$.
\item
$d\circ (\pi_{(S,0)})_{s\in S}=\pi_{(S,0)}$, where $0\in S$.
\end{enumerate}
\end{prop}
\begin{proof}
1. 
\begin{align*}
e_s\circ d
=&e_s\circ (\pi_{(S\times S,p_2,(t,t))}^M)_{t\in S}&\\
=&e_s\circ (\pi_{(S\times S,p_2,(s,t))}^M)_{t\in S}&(e_s\text{ depends only on }s\text{-componet})\\
=&e_s\circ \pi_{(S,s)}^{{\rm H}(M)}.&
\end{align*}
Here and through the proof of this proposition,
$\pi^M$ and $\pi^{{\rm H}(M)}$ denote projection constants
of $M$ and ${\rm H}(M)$ respectively.

2.
\begin{align*}
d\circ (e_s\circ \pi_{(S,s)}^{{\rm H}(M)})_{s\in S}
=& (\pi_{(S\times S,p_2,(s,s))}^M)_{s\in S}
 \circ ((e_{s,t})_{t\in S}\circ (\pi_{(S\times S,p_2,(s,u))}^M)_{u\in S})_{s\in S}\\
=& (e_{s,s}\circ (\pi_{(S\times S,p_2,(s,u))}^M)_{u\in S})_{s\in S}\\
=& (\pi_{(S\times S,p_2,(s,s))}^M)_{s\in S}\\
=& d.
\end{align*}

3.
\begin{align*}
d\circ (\pi_{(S,0)}^{{\rm H}(M)})_{s\in S}
=&(\pi_{(S\times S,p_2,(s,s))}^M)_{s\in S}
 \circ (\pi_{(S\times S,p_2,(0,t))}^M)_{(s,t)\in S\times S}\\
=& (\pi_{(S\times S,p_2,(0,s))}^M)_{s\in S}\\
=& \pi_{(S,0)}^{{\rm H}(M)}.
\end{align*}
\end{proof}

We give a name for a tuple of these terms.
\begin{df}\label{diagonal-pair}
Let $C$ be a single-sorted $<\mc\kappa$-ary clone.
The tuple $(d,(e_s)_{s\in S})\in C_S\times C_1{}^S$ is 
said an $S$-ary diagonal pair of $C$ if the following equations holds:
\begin{enumerate}
\item
$e_{s}\circ d=\pi_{(S,s)}$ for all $s\in S$.
\item
$d(e_s\circ \pi_{(S,s)})_{s\in S}=d$.
\item
$d(\pi_{(S,0)})_{s\in S}=\pi_{(S,0)}$ for $0\in S$.
\end{enumerate}
\end{df}

\begin{rem}
If $(d,(e_s)_{s\in S})$ is a diagonal pair of a clone $C$,
then $d$ satisfies diagonal identity
\[
d(d(\pi_{(S\times S,(s,t))})_{t\in S})_{s\in S}
=d(\pi_{(S\times S,(s,s))})_{s\in S}.
\]
Namely, $C$-algebras have diagonal algebras reduct.
\end{rem}
\begin{proof}
\begin{align*}
d(d(\pi_{(S\times S,(s,t))})_{t\in S})_{s\in S}
=&d(e_sd(\pi_{(S\times S,(s,t))})_{t\in S})_{s\in S}\\
=&d(e_s\pi_{(S\times S,(s,s))})_{s\in S}\\
=&d(\pi_{(S\times S,(s,s))})_{s\in S}.
\end{align*}
\end{proof}

\section{Matrix product} \label{s-matrix-product}

A homogenization ${\rm H}(M)$ of an $S$-sorted pure clone has a diagonal pair.
A clone that has a diagonal pair is also characterized 
as being decomposable with respect to matrix product.
In this section, we explain these conditions are equivalent.
In the next section, we explain these conditions also be equivalent to
the following condition: There exists a many-sorted pure clone $M$ such that
the homogenization of $M$ is isomorphic to the single-sorted clone.

We start from the definition of matrix product.
\begin{df}
Let $C$ be a single-sorted $<\mc\kappa$-clone.
\begin{enumerate}
\item
Let $e\in C_1$ be an idempotent element, namely, $e$ satisfies $e\circ e=e$.
An idempotent retract, denoted by $e(C)$, of $C$ by $e$ is the following $<\mc\kappa$-ary clone.
\begin{itemize}
\item (Underlying set)
The underlying set $e(C)_{\lambda}$ of the sort $\lambda<\kappa$ is 
$\{f\in C_\lambda\mid e\circ f=f\}/\mc\sim_\lambda$,
where $\sim_\lambda$ is the equivalence relation defined as
$f\sim_\lambda g :\Leftrightarrow
f\circ (e\pi_{(\lambda,i)})_{i\in\lambda}=g\circ (e\pi_{(\lambda,i)})_{i\in\lambda}$.
\item (Projection)
The $\lambda$-ary $i$-th projection $\pi^{e(C)}_{(\lambda,i)}$
is $e\circ \pi^C_{(\lambda,i)}\mc/\mc\sim_\lambda$, where $\pi^C_{(\lambda,i)}$ is
the $\lambda$-ary $i$-th projection of $C$.
\item (Composition)
For $f\mc/\mc\sim_{\lambda_1}\in e(C)_{\lambda_1}$ and
$g_i\mc/\mc\sim_{\lambda_2}\in C(C)_{\lambda_2}$ ($i\in\lambda_1$), we define
\[
c^{e(C)}(f\mc/\mc\sim_{\lambda_1},(g_i\mc/\mc\sim_{\lambda_2})_{i\in\lambda_1})
:=c^C(f,(g_i)_{i\in\lambda_1})\mc/\mc\sim_{\lambda_2},
\]
where $c^C$ is the composition operation of $C$.
\end{itemize}
\item
Let $(e_s)_{s\in S}\in C_1{}^S$ be a family of idempotent elements.
The matrix product $\boxtimes_{s\in S}e_s(C)$ of a family of
idempotent retracts $(e_s(C))_{s\in S}$ is defined as
the following $<\mc\kappa$-ary clone:
\begin{itemize}
\item (Underlying set)
The underlying set $(\boxtimes_{s\in S}e(C))_{\lambda}$ of the sort $\lambda<\kappa$ is 
\[
\{(f_s/\mc\sim_{\lambda})_{s\in S}\in (C_{\lambda\times S}/\mc\sim_{\lambda})^S
 \mid \forall s\in S;e_sf_s=f_s\}, 
\]
where the equivalence $\sim_\lambda$ is defined as
\[
f\sim_{\lambda} g:\Longleftrightarrow
f\circ (e_t\pi_{(\lambda\times S,(i,t))})_{(i,t)\in\lambda\times S}
=g\circ (e_t\pi_{(\lambda\times S,(i,t))})_{(i,t)\in\lambda\times S}.
\]
\item (Projection)
The $\lambda$-ary $i$-th projection $\pi_{(\lambda,i)}$ is defined by
$(e_s\circ \pi^{C}_{\lambda\times S,(i,s)}\mc/\mc\sim_\lambda)_{s\in S}$,
where $\pi_{(\lambda\times S,(i,s))}^C$ is the $\lambda\times S$-ary $(i,s)$-th
projection constant of $C$.
\item (Composition)
For $f=(f_s\mc/\mc\sim_{\lambda_1})_{s\in S}\in (\boxtimes_{s\in S}e(C))_{\lambda_1}$
and $g_i=(g_{i,s}\mc/\mc\sim_{\lambda_2})_{s\in S}\in (\boxtimes_{s\in S}e(C))_{\lambda_2}$
($i\in\lambda_1$), the composition is defined by
\[
c(f,(g_i)_{i\in\lambda_1})
:=(c^C(f_s,(g_{i,t})_{(i,t)\in\lambda_1\times S})\mc/\mc\sim_{\lambda_2})_{s\in S},
\]
where $c^C$ is the composition operation of $C$.
\end{itemize}
\end{enumerate}
\end{df}

\begin{df}[cf.\ {\cite[Definition 2.4, Lemma 3.5]{Kea}}]
Let $C$ be a $<\mc\kappa$-ary clone, $A$ be a $C$-algebra.
\begin{enumerate}
\item
Let $e\in C_1$ be an idempotent element.
The idempotent retract (in literature e.g. \cite{Iza},\cite{Kea},
this is referred as a neighbourhood) of 
$A$ by $e$, denoted by $e(A)$, is defined as the following $e(C)$-algebra:
\begin{itemize}
\item
The underlying set is $e(A)=\{e(a)\mid a\in A\}$.
\item
The action of $f\mc/\mc\sim_\lambda\in (e(C))_{\lambda}$ is
$(a_i)_{i\in\lambda}\mapsto f(a_i)_{i\in\lambda}$.
\end{itemize}
\item
Let $(e_s)_{s\in S}$ be a family of idempotent elements of $C$.
The matrix product $\boxtimes_{s\in S}e_s(A)$ is defined as
the following $\boxtimes_{s\in S}e_s(C)$-algebra:
\begin{itemize}
\item
The underlying set is $\prod_{s\in S}e_s(A)$.
\item
The action of
$(f_s\mc/\mc\sim_\lambda)_{s\in S}\in (\boxtimes_{s\in S}e_s(C))_\lambda$ is
\[
((a_{i,t})_{t\in S})_{i\in\lambda}\mapsto 
(f_s(a_{i,t})_{(i,t)\in\lambda\times S})_{s\in S}.
\]
\end{itemize}
\end{enumerate}
\end{df}


The mapping $A\mapsto \boxtimes_{s\in S}e_s(A)$ can be extended as a functor
from the category of $C$-algebras to the category of $\boxtimes_{s\in S}e_s(C)$-algebras.
Precisely, the following statement hold.
\begin{prop}
Let $C$ be a $<\mc\kappa$-ary clone and $(e_s)_{s\in S}$ be a
family of idempotent elements of $C$.
Let $A$ and $B$ be $C$-algebras, $\ph:A\ra B$ be a homomorphism of $C$-algebras.
Then the mapping \mbox{$\tilde{\ph}:\boxtimes_{s\in S}e_s(A)\ra \boxtimes_{s\in S}e_s(B)$}
defined as $(a_s)_{s\in S}\mapsto (\ph(a_s))_{s\in S}$ is a homomorphism of
$\boxtimes_{s\in S}e_s(C)$-algebras.
Furthermore, the (pair of) correspondence $A\mapsto \boxtimes_{s\in S}e_s(A)$ and
$\ph\mapsto \tilde{\ph}$ is a functor $\V(C)\ra \V(\boxtimes_{s\in S}e_s(C))$.
\end{prop}
\begin{proof}
Let $f=(f_s\mc/\mc\sim_\lambda)_{s\in S}\in (\boxtimes_{s\in S}e_s(C))_{\lambda}$ and
$a_{i}=(a_{i,s})_{s\in S}\in \boxtimes_{s\in S}e_s(A)$ for $i\in\lambda$.
Then
\begin{align*}
\tilde{\ph}(f(a_i)_{i\in\lambda})
&=(\ph(f_s(a_{i,t})_{(i,t)\in\lambda\times S}))_{s\in S}\\
&=(f_s(\ph(a_{i,t}))_{(i,t)\in\lambda\times S})_{s\in S}\\
&=f(\tilde{\ph}(a_i))_{i\in\lambda}.
\end{align*}
The correspondence $\ph\mapsto \tilde{\ph}$ clearly preserves
identity morphisms and composition.
Thus the correspondence is a functor.
\end{proof}

Matrix product has a diagonal pair.
A diagonal pair of a matrix product is, for example,
constructed as in the next proposition.
\begin{prop}
Let $C$ be a single-sorted $<\mc\kappa$-ary clone, $(e_s)_{s\in S}$ be
a family of idempotent elements of $C$.
Let $d\in (\boxtimes_{s\in S} e_s(C))_{S}$ and 
$\tilde{e}_s\in (\boxtimes_{s\in S} e_s(C))_1$ be the terms defined as follows:
\begin{align*}
d&:=(e_s\circ \pi_{(S\times S,(s,s))}\mc/\mc\sim_S)_{s\in S},\\
\tilde{e}_s&:=(e_t\circ \pi_{(S,s)}\mc/\mc\sim_1)_{t\in S}.
\end{align*}
Then $(d,(\tilde{e}_s)_{s\in S})$ is a diagonal pair of
$\boxtimes_{s\in S}e_s(C)$.
\end{prop}

\begin{proof}
$d,\tilde{e}_s\in\boxtimes_{t\in S} e_t(C)$ immediately follows from the definition.
In the proof of this proposition, we write 
$\pi^C$/$\pi^{\boxtimes e(C)}$ the projection terms of
$C$/$\boxtimes_{s\in S}e_s(C)$ respectively.

1. 
\begin{align*}
\tilde{e}_s\circ d
=& \tilde{e}_s\circ (e_t\circ \pi_{(S\times S,(t,t))}^C)_{t\in S}&\\
=& \tilde{e}_s\circ (e_t\circ \pi_{(S\times S,(s,t))}^C)_{t\in S}&
  (\tilde{e}_s\text{ depends only on }s\text{-component})\\
=&\tilde{e}_s\circ \pi_{(S,s)}^{\boxtimes e(C)}.&
\end{align*}

2. 
\begin{align*}
d\circ (\tilde{e}_s\circ \pi_{(S,s)}^{\boxtimes e(C)})_{s\in S}
=& (e_t\circ \pi_{(S\times S,(t,t))}^C)_{t\in S}
  \circ ((e_t\pi_{(S,s)}^C)_{t\in S}
          \circ (\pi_{(S\times S,(s,t))}^C)_{t\in S})_{s\in S}\\
=& (e_t\circ \pi_{(S\times S,(t,t))}^C)_{t\in S}
  \circ ((e_t\pi_{(S\times S,(s,s))}^C)_{t\in S})_{s\in S}\\
=& (e_t\circ e_t\pi_{(S\times S,(t,t))}^C)_{t\in S}\\
=& (e_t\circ \pi_{(S\times S,(t,t))}^C)_{t\in S}\\
=& d.
\end{align*}

3.
\begin{align*}
d\circ (\pi_{(S,0)}^{\boxtimes e(C)})_{s\in S}
=& (e_t\circ \pi_{(S\times S,(t,t))}^C)_{t\in S}
  \circ ((e_t\circ \pi_{(S\times S,(0,t))}^C)_{t\in S})_{s\in S}\\
=& (e_t\circ (e_t\circ \pi_{(S\times S,(0,t))}^C))_{t\in S}\\
=& (e_t\circ \pi_{(S\times S,(0,t))}^C)_{t\in S}\\
=& \pi_{(S,0)}^{\boxtimes e(C)}.
\end{align*}
\end{proof}

Conversely, a clone that has a diagonal pair
is decomposed with respect to matrix product.
In this sense, matrix product is characterized by existence
of a diagonal pair.

\begin{thm}
Let $C$ be a single-sorted clone, $(d,(e_s)_{s\in S})$ be a diagonal pair.
Then the clone $C$ is isomorphic to $\boxtimes_{s\in S}e_s(C)$.
\end{thm}
\begin{proof}
Suppose $(d,(e_s)_{s\in S})$ is a diagonal pair of $C$.
We define $\ph_\lambda:C_\lambda\ra (\boxtimes_{s\in S}e_s(C))_\lambda$
as
\[
\ph_\lambda:f\mapsto 
 (e_s\circ f\circ 
  (d\circ(\pi_{(\lambda\times S,(i,t))})_{t\in S})
   _{i\in\lambda}/\mc\sim_\lambda)_{s\in \lambda}.
\]
Note that $\ph_{\lambda}(f)\in (\boxtimes_{s\in S}e_s(C))_\lambda$
follows from $e_s\circ e_s=e_s$.

We prove injectivity of $\ph_\lambda$.
Let $f,g\in C_\lambda$ and assume $\ph_\lambda(f)=\ph_\lambda(g)$,
namely, assume the equality
\begin{align*}
 &e_sf\circ(d\circ(e_t\pi_{(\lambda\times S,(i,t))})_{t\in S})_{i\in\lambda}\\
=&e_sf\circ (d\circ (\pi_{(\lambda\times S,(i,t))})_{t\in S})_{i\in\lambda}
  \circ (e_t\pi_{(\lambda\times S,(i,t))})_{(i,t)\in\lambda\times S}\\
=&e_sg\circ (d\circ (\pi_{(\lambda\times S,(i,t))})_{t\in S})_{i\in\lambda}
  \circ (e_t\pi_{(\lambda\times S,(i,t))})_{(i,t)\in\lambda\times S}\\
=&e_sg\circ(d\circ(e_t\pi_{(\lambda\times S,(i,t))})_{t\in S})_{i\in\lambda}
\end{align*}
hold for all $s\in S$. Then
\begin{align*}
f
=&d\circ (e_s)_{s\in S}\circ f\circ (\pi_{(\lambda,i)})_{i\in\lambda}\\
=&d\circ (e_s)_{s\in S}\circ f\circ 
  (d\circ (\pi_{(\lambda,i)})_{t\in S})_{i\in\lambda}\\
=&d\circ (e_s)_{s\in S}\circ f\circ 
  (d\circ (\pi_{(\lambda\times S,(i,t))})_{t\in S})_{i\in\lambda}
  \circ (\pi_{(\lambda,i)})_{(i,t)\in\lambda\times S}\\
=&d\circ (e_s)_{s\in S}\circ f\circ 
  (d\circ (e_t\pi_{(\lambda\times S,(i,t))})_{t\in S})_{i\in\lambda}
  \circ (\pi_{(\lambda,i)})_{(i,t)\in\lambda\times S}\\
=&d\circ (e_s)_{s\in S}\circ g\circ 
  (d\circ (e_t\pi_{(\lambda\times S,(i,t))})_{t\in S})_{i\in\lambda}
  \circ (\pi_{(\lambda,i)})_{(i,t)\in\lambda\times S}\\
=&g.
\end{align*}

Next, we prove surjectivity of $\ph_\lambda$.
Let $(f_s\mc/\mc\sim_{\lambda})_{s\in S}\in(\boxtimes_{s\in S}e_s(C))_\lambda$.
We define $\tilde{f}\in C_\lambda$ as
\[
\tilde{f}:=d\circ (f_s)_{s\in S}\circ 
(e_t\pi_{(\lambda,i)})_{(i,t)\in\lambda\times S}.
\]
Then
\begin{align*}
\ph_{\lambda}(\tilde{f})
=&\left(e_sd\circ (f_t)_{t\in S}
  \circ (e_t\pi_{(\lambda,i)})_{(i,t)\in\lambda\times S}
  \circ (d\circ (\pi_{(\lambda\times S,(i,u))})_{u\in S})_{i\in\lambda}\mc/\mc\sim_{\lambda}\right)_{s\in S}\\
=&\left(e_sf_s
  \circ (e_td\circ (\pi_{(\lambda\times S,(i,u))})_{u\in S})_{(i,t)\in\lambda\times S}\mc/\mc\sim_{\lambda}\right)_{s\in S}\\
=&\left(f_s
  \circ (e_t\pi_{(\lambda\times S,(i,t))})_{(i,t)\in\lambda\times S}\mc/\mc\sim_{\lambda}\right)_{s\in S}\\
=&(f_s\mc/\mc\sim_{\lambda})_{s\in S}.
\end{align*}

Finally, we prove $\ph=(\ph_{\lambda})_{\lambda<\kappa}$ is a homomorphism.
Let $\lambda_1,\lambda_2<\kappa$ and $f\in C_{\lambda_1}$,
$g_i\in C_{\lambda_2}$ for $i\in\lambda_1$.
Then
\begin{align*}
&\ph_{\lambda_2}(f\circ (g_i)_{i\in\lambda_1})\\
=&\left(e_s f\circ (g_i)_{i\in\lambda_1}
  \circ (d\circ (\pi_{(\lambda_2\times S,(j,t))})_{t\in S})_{j\in\lambda_2}\mc/\mc\sim_{\lambda_2}\right)_{s\in S}\\
=&\left(e_sf\circ (d\circ (g_i)_{u\in S})_{i\in\lambda_1}
  \circ (d\circ (\pi_{(\lambda_2\times S,(j,t))})_{t\in S})_{j\in\lambda_2}\mc/\mc\sim_{\lambda_2}\right)_{s\in S}\\
=&\left(e_s f\circ (d\circ (e_ug_i)_{u\in S})_{i\in\lambda_1}
  \circ (d\circ (\pi_{(\lambda_2\times S,(j,t))})_{t\in S})_{j\in\lambda_2}\mc/\mc\sim_{\lambda_2}\right)_{s\in S}\\
=&\left(e_s f\circ (d\circ (\pi_{(\lambda_1\times S,(i,u))})_{u\in S})_{i\in\lambda_1}
  \circ (e_ug_i)_{(i,u)\in\lambda_1\times S}
  \circ (d\circ (\pi_{(\lambda_2\times S,(j,t))})_{t\in S})_{j\in\lambda_2}\mc/\mc\sim_{\lambda_2}\right)_{s\in S}\\
=&\left(e_s f\circ (d\circ (\pi_{(\lambda_1\times S,(i,u))})_{u\in S})_{i\in\lambda_1}\mc/\mc\sim_{\lambda_1}\right)_{s\in S}
  \circ \left(e_ug_i
  \circ (d\circ (\pi_{(\lambda_2\times S,(j,t))})_{t\in S})_{j\in\lambda_2}\mc/\mc\sim_{\lambda_2}\right)_{(i,u)\in \lambda_1\times S}\\
=&\ph_{\lambda_1}(f)\circ (\ph_{\lambda_2}(g_i))_{i\in\lambda_1}.
\end{align*}
\end{proof}

\section{Heterogenization} \label{s-heterogenization}

Let $M$ be an $S$-sorted pure clone.
Then the homogenization ${\rm H}(M)$ has an $S$-ary diagonal pair.
The converse, if a clone $C$ has an $S$-ary diagonal pair
then there exists an $S$-sorted pure clone $M$ such that ${\rm H}(M)$
is isomorphic to $C$, is also true.
In this section, we explain the construction of the many-sorted clone $M$ from a single-sorted clone $C$.
The isomorphism is proved in the next section.

\begin{df}\label{heterogenization-clone}
Let $C$ be a single-sorted $<\mc\kappa$-ary clone. Assume $C$ has an $S$-ary diagonal pair $(d,(e_s)_{s\in S})$.
We define an $S$-sorted $<\mc\kappa$-ary clone $C_{(d,(e_s)_{s\in S})}$ as follows:
\begin{itemize}
\item (Underlying set)
For $\lambda<\kappa$, $v:\lambda\ra S$ and $s\in S$,
the underlying set of the sort $(\lambda,v,t)$
is
\[
C_{(d,(e_s)_{s\in S}),(\lambda,v,t)}
:=\{f\in C_\lambda\mid
 e_tf=f\}/\mc\sim_{\lambda,v},
\]
where the equivalence $\sim_{\lambda,v}$ is defined as
\begin{align*}
f\sim_{\lambda,v} g &:\Longleftrightarrow 
f\circ (e_{v(i)}\circ \pi_{(\lambda,i)})_{i\in\lambda}
=g\circ (e_{v(i)}\circ \pi_{(\lambda,i)})_{i\in\lambda}.
\end{align*}
\item (Projection)
For $\lambda<\kappa$, $v:\lambda\ra S$ and $i\in\lambda$,
the projection constant $\pi_{(\lambda,v,i)}$ is defined as
$e_{v(i)}\circ \pi_{(\lambda,i)}\mc/\mc\sim_{\lambda,v}$.
\item (Composition)
For $\lambda_k<\kappa$, $v:\lambda_k\ra S$ ($k=1,2$) and $s\in S$,
the composition operation
\begin{align*}
c_{(\lambda_1,v_1,s),(\lambda_2,v_2)}:&
C_{(d,(e_s)_{s\in S}),(\lambda_1,v_1,s)}\times 
\prod_{i\in \lambda_1}C_{((d,(e_s)_{s\in S}),(\lambda_2,v_2,v_1(i))}\\
\ra & C_{(d,(e_s)_{s\in S}),(\lambda_2,v_2,s)}
\end{align*}
is defined by
$(f/\mc\sim_{\lambda_1,v_1},(g_i/\mc\sim_{\lambda_2,v_2})_{i\in\lambda_1})
\mapsto (f\circ(g_i)_{i\in\lambda_1})/\mc\sim_{\lambda_2,v_2}$.
\end{itemize}
\end{df}

\begin{df}
Let $C$ be a $<\mc\kappa$-ary clone that has a diagonal pair $(d,(e_s)_{s\in S})$.
Let $A$ be a $C$-algebra.
We define a $C_{(d,(e_s)_{s\in S})}$-algebra $A_{(d,(e_s)_{s\in S})}$ as follows:
\begin{itemize}
\item (Underlying set)
The underlying set $A_{(d,(e_s)_{s\in S}),s}$ of the sort $s\in S$ is $e_s(A)$.
\item (Action of $C_{(d,(e_s)_{s\in S})}$)
A term $f/\mc\sim_{\lambda,v}\in C_{(d,(e_s)_{s\in S}),(\lambda,v,s)}$ acts as
\[
\prod_{i\in\lambda} A_{(d,(e_s)_{s\in S}),v(i)}\ni(a_i)_{i\in\lambda}
\mapsto f(a_i)_{i\in\lambda}\in A_{(d,(e_s)_{s\in S}),s}.
\]
\end{itemize}
\end{df}

\begin{rem}\label{hetero-representative}
In the Definition \ref{heterogenization-clone},
$g:=f\circ (e_{v(i)}\circ \pi_{(\lambda,i)})_{i\in\lambda}\sim_{\lambda,v} f$
holds for $f\in C_\lambda$. Moreover $g$ satisfy
$g\circ (e_{v(i)}\circ \pi_{(\lambda,i)})_{i\in\lambda}=g$.
Thus for any element $\tilde{f}$ of $C_{(d,(e_s)_{s\in S}),(\lambda,v,t)}$,
we can choose a term $g\in C_\lambda$ that satisfies
$g\circ (e_{v(i)}\circ \pi_{(\lambda,i)})_{i\in\lambda}=g$
as a representative for $\tilde{f}$.
\end{rem}

The structure of the $S$-sorted clone $C_{(d,(e_s)_{s\in S})}$ and 
the algebra $A_{(d,(e_s)_{s\in S})}$
defined above do not
depend on the family of projection idempotents $(e_s)_{s\in S}$.
Next, we prove this fact.
We start from the following lemma.
\begin{lem}\label{different-projection}
Let $C$ be a single-sorted $<\mc\kappa$-ary clone and ${(d,(e_s)_{s\in S})}$, ${(d,(e'_s)_{s\in S})}$
be diagonal pairs of $C$.
Then $e_se'_s=e_s$ holds for $s\in S$.
\end{lem}
\begin{proof}
\[
e_se'_s=e_sd(e'_t\pi_{(\{0\},0)})_{t\in S}=e_sd(\pi_{(\{0\},0)})_{t\in S}=e_s.
\qedhere
\]
\end{proof}
\begin{prop}
Let $C$ be a $<\mc\kappa$-ary clone and ${(d,(e_s)_{s\in S})}$, ${(d,(e'_s)_{s\in S})}$
be diagonal pairs of $C$.
Then $C_{(d,(e_s)_{s\in S})}$ is isomorphic to  $C_{(d,(e'_s)_{s\in S})}$ via an isomorphism
$\ph=(\ph_{(\lambda,v,s)})$ where
\[
\ph_{(\lambda,v,s)}:C_{(d,(e_s)_{s\in S}),(\lambda,v,s)}\ni f/\mc \sim\mapsto 
e'_sf(e_{v(i)}\pi_{(\lambda,i)})_{i\in\lambda}/\mc\sim'
\in C_{(d,(e'_s)_{s\in S}),(\lambda,v,s)}.
\]
Here, $\sim$ and $\sim'$ are equivalences defined as
\begin{align*}
f\sim g:\Longleftrightarrow 
  f\circ (e_{v(i)}\circ \pi_{(\lambda,i)})_{i\in\lambda}
 =g\circ (e_{v(i)}\circ \pi_{(\lambda,i)})_{i\in\lambda},\\
f\sim' g:\Longleftrightarrow 
  f\circ (e'_{v(i)}\circ \pi_{(\lambda,i)})_{i\in\lambda}
 =g\circ (e'_{v(i)}\circ \pi_{(\lambda,i)})_{i\in\lambda}.
\end{align*}
\end{prop}
\begin{proof}
Well-definedness of $\ph_{(\lambda,v,s)}$ directly follows from the definition.

Let $\psi_{(\lambda,v,s)}:C_{(d,(e'_s)_{s\in S}),(\lambda,v,s)}\ra
C_{(d,(e_s)_{s\in S}),(\lambda,v,s)}$ be the mapping
$g/\mc\sim'\mapsto e_sg(e'_{v(i)}\pi_{(\lambda,i)})_{i\in\lambda}/\mc\sim$.
Then by the previous lemma,
\begin{align*}
\psi_{(\lambda,v,s)}\circ \ph_{(\lambda,v,s)}(f/\mc\sim)
=&\left(e_se'_sf(e_{v(i)}\pi_{(\lambda,i)})_{i\in\lambda}
  \circ(e'_{v(i)}\pi_{(\lambda,i)})_{i\in\lambda}\right)/\mc\sim\\
=&\left(e_sf(e_{v(i)}\pi_{(\lambda,i)})_{i\in\lambda}\right)/\mc\sim\\
=&f/\mc\sim.
\end{align*}
By the same way,
$\ph_{(\lambda,v,s)}\circ \psi_{(\lambda,v,s)}=
\id_{C_{(d,(e'_s)_{s\in S}),(\lambda,v,s)}}$
holds. Thus $\psi_{(\lambda,v,s)}$ is the inverse of $\ph_{(\lambda,v,s)}$.
Particularly, $\ph_{(\lambda,v,s)}$ is bijective.

We prove $\ph$ is compatible with $S$-sorted clone operations.
Let $\pi_{(\lambda,v,i)}=e_{v(i)}\pi_{(\lambda,i)}/\mc\sim$ be the
$(\lambda,v)$-ary $i$-th projection of $C_{(d,(e_s)_{s\in S})}$.
Then
\[
\ph_{(\lambda,v,v(i))}(\pi_{(\lambda,v,i)})
=\left(e'_{v(i)}e_{v(i)}\pi_{(\lambda,i)}(e_{v(j)}\pi_{(\lambda,j)})_{j\in\lambda}\right)/\mc\sim'
=(e'_{v(i)}e_{v(i)}\circ e_{v(i)}\pi_{(\lambda,i)})/\mc\sim'
=e'_{v(i)}\pi_{(\lambda,i)}/\mc\sim'
\]
by the previous lemma.
It is the $(\lambda,v)$-ary $i$-th projection of $C_{(d,(e_s)_{s\in S})}$.

Next, let $\lambda_k<\kappa$, $v_k:\lambda_k\ra S$ ($k=1,2$),
$s\in S$ and 
\[
f/\mc\sim\in C_{(d,(e_s)_{s\in S}),(\lambda_1,v_1,s)},
g_i/\mc\sim\in C_{(d,(e_s)_{s\in S}),(\lambda_2,v_2,v_1(i))}
\]
for $i\in\lambda_1$.
Then
\begin{align*}
&\ph_{(\lambda_1,v_1,s)}(f/\mc\sim)
 \circ (\ph_{(\lambda_2,v_2,v_1(i))}(g_i/\mc\sim))_{i\in\lambda_1}&\\
=&(e'_sf(e_{v_1(i)}\pi_{(\lambda_1,i)})_{i\in\lambda_1}/\mc\sim)\circ 
 (e'_{v_1(i)}g_i(e_{v_2(j)}\pi_{(\lambda_2,j)})_{j\in\lambda_1}/\mc\sim)_{i\in\lambda_1}&\\
=&\left(e'_sf(e_{v_1(i)}e'_{v_1(i)}g_i)_{i\in\lambda_1}\circ 
  (e_{v_2(j)}\pi_{(\lambda_2,j)})_{j\in\lambda_2}\right)/\mc\sim'&\\
=&\left(e'_sf(e_{v_1(i)}g_i)_{i\in\lambda_1}\circ 
  (e_{v_2(j)}\pi_{(\lambda_2,j)})_{j\in\lambda_2}\right)/\mc\sim'
  &(\text{Lemma \ref{different-projection}})\\
=&\left(e'_sf(g_i)_{i\in\lambda_1}\circ 
  (e_{v_2(j)}\pi_{(\lambda_2,j)})_{j\in\lambda_2}\right)/\mc\sim'
  &(g_i/\mc\sim\in C_{(d,(e_s)_{s\in S})})\\
=&\ph_{(\lambda_2,v_2,s)}(f\circ (g_i)_{i\in\lambda_1}).&
\end{align*}
\end{proof}

Under the identification by this isomorphism, the $S$-sorted algebras 
by $(d,(e_s)_{s\in S})$ and $(d,(e'_s)_{s\in S})$ are isomorphic.
\begin{prop}
Let $C$ be a single-sorted $<\mc\kappa$-ary clone, ${(d,(e_s)_{s\in S})}$, ${(d,(e'_s)_{s\in S})}$
be diagonal pairs of $C$ and \mbox{$A=(A_s)_{s\in S}$} be a $C$-algebra.
Then $A_{(d,(e_s)_{s\in S})}$ and $A_{(d,(e'_s)_{s\in S})}$ are isomorphic via
$\ph=(\ph_t)_{t\in S}$ where
\[
\ph_s:A_{(d,(e_s)_{s\in S}),t}\ni a\mapsto e'_t(a)\in A_{(d,(e'_s)_{s\in S}),t}.
\]
\end{prop}
\begin{proof}
By Lemma \ref{different-projection},
\[
a=e'_te_t(a), a'=e_te'_t(a')
\]
hold for $a\in A_{(d,(e_s)_{s\in S}),t}$ and $a'\in A_{(d,(e_s)_{s\in S}),t}$.
Thus $a'\mapsto e_t(a')$ is the inverse of $\ph_t$ and $\ph_t$ is bijective.

Let $f/\mc\sim\in C_{(d,(e_s)_{s\in S}),(\lambda,v,t)}$ and $a_i\in A_{(d,(e_s)_{s\in S}),v(i)}$.
Then
\begin{align*}
\ph_t((f/\mc\sim)(a_i)_{i\in \lambda})
&=e'_tf(a_i)_{i\in\lambda}\\
&=e'_tf(e_{v(i)}(a_i))_{i\in\lambda}\\
&=e'_tf(e_{v(i)}e'_{v(i)}(a_i))_{i\in\lambda}\\
&=e'_tf\circ(e_{v(i)}\pi_{(\lambda,i)})_{i\in\lambda}(e'_{v(i)}(a_i))_{i\in\lambda}\\
&=\ph_{(\lambda,v,t)}(f/\mc\sim)(\ph_{v(i)}(a_i))_{i\in\lambda}.
\end{align*}
This equation means $\ph=(\ph_t)_{t\in S}$ is a homomorphism
$A_{(d,(e_s)_{s\in S})}$ to $A_{(d,(e'_s)_{s\in S})}$.
\end{proof}

In the following paragraph, we omit projection idempotents for writing heterogenization.
Namely, we simply write $C_{d}$ the $S$-sorted clone $C_{(d,(e_s)_{s\in S})}$.

We end this section to prove the correspondence $A\mapsto A_d$ above is extended to a functor.
\begin{prop}\label{single-many}
Let $C$ be a single-sorted $<\mc\kappa$-ary clone, 
$(d,(e_s)_{s\in S})$ be a diagonal pair of $C$, $\ph:A\ra B$ be a homomorphism of $C$-algebras.
Then $\ph_d:A_d\ra B_d$ defined as $\ph_{d,s}(a):=\ph(a)$ for $s\in S,a\in A_{d,s}$
is a homomorphism of $C_d$-algebras.
Further, the correspondence $d$ ($A\mapsto A_d$ for $A\in\V(C)$ and
$\ph\mapsto \ph_d$ for morphisms of $C$-algebras) is a functor
$\Cat(\V(C))\ra \Cat(\V(C_d))$.
\end{prop}
\begin{proof}
At first, notice that $a\in A_{d,s}=e_s(A)$ implies
$\ph(a)=\ph(e_s(a))=e_s(\ph(a))\in e_s(B)=B_s$.
Thus $\ph_{d,s}$ is actually a mapping $A_s\ra B_s$.

Let $f\in C_{d,(\lambda,v,s)}$, $a_i\in A_{d,v(i)}$.
Then
\begin{align*}
\ph_{d,s}((f/\mc\sim)(a_i)_{i\in\lambda})
&=\ph f(a_i)_{i\in\lambda}\\
&=f(\ph (a_i))_{i\in\lambda}\\
&=(f/\mc\sim)(\ph_{d,v(i)} (a_i))_{i\in\lambda}.
\end{align*}
This equation means nothing but $\ph_d$ is a homomorphism.

It is easily verified that $d$ preserves identity and composition,
thus $d$ is a functor.
\end{proof}

\section{Concrete equivalence} \label{s-conrete-equivalence}

The aim of this section is proving that the functors described in 
Proposition \ref{many-single} and \ref{single-many} are
mutually inverse categorical equivalence.

\begin{prop}\label{clone-single}
Let $C$ be a single-sorted $<\mc\kappa$-ary clone that has a diagonal pair $(d,(e_s)_{s\in S})$.
Then ${\rm H}(C_d)$ and $C$ are isomorphic via an isomorphism $\nu=\{\nu_\lambda\}_{\lambda<\kappa}$
\[
\nu_{\lambda}:C_\lambda\ni f\mapsto
(e_sf\circ (d\circ (\pi_{(\lambda\times S,(i,t))})_{t\in S})_{i\in \lambda}/\mc\sim)_{s\in S}
 \in {\rm H}_\lambda(C_d),
\]
where $\sim$ is the following equivalence relation on $C_{\lambda\times S}$:
\[
f_s\sim g_s\ \Longleftrightarrow \
  f_s\circ (e_t\circ \pi_{(\lambda\times S,(i,t))})_{(i,t)\in\lambda\times S}
 =g_s\circ (e_t\circ \pi_{(\lambda\times S,(i,t))})_{(i,t)\in\lambda\times S}.
\]
\end{prop}
\begin{proof}
(Injectivity) Let $f,g\in C_\lambda$ and assume $\nu_\lambda(f)=\nu_\lambda(g)$,
namely
\begin{align*}
&e_sf\circ (d\circ (\pi_{(\lambda\times S,(i,t))})_{t\in S})_{i\in\lambda}
  \circ (e_t\circ \pi_{(\lambda\times S,(i,t))})_{(i,t)\in\lambda\times S}\\
=&e_sg\circ (d\circ (\pi_{(\lambda\times S,(i,t))})_{t\in S})_{i\in\lambda}
  \circ (e_t\circ \pi_{(\lambda\times S,(i,t))})_{(i,t)\in\lambda\times S}
\end{align*}
holds for all $s\in S$. Then
\begin{align*}
f
&=d\circ (e_sf)_{s\in S}\\
&=d\circ (e_sf)_{s\in S}\circ (\pi_{(\lambda,i)})_{i\in\lambda}\\
&=d\circ (e_sf)_{s\in S}\circ (d\circ (\pi_{(\lambda,i)})_{t\in S})_{i\in\lambda}\\
&=d\circ (e_sf)_{s\in S}\circ 
  (d\circ (\pi_{(\lambda\times S,(i,t))})_{t\in S})_{i\in\lambda}
  \circ (\pi_{(\lambda,i)})_{(i,t)\in\lambda\times S}\\
&=d\circ (e_sf)_{s\in S}\circ 
  (d\circ (e_t\pi_{(\lambda\times S,(i,t))})_{t\in S})_{i\in\lambda}
  \circ (\pi_{(\lambda,i)})_{(i,t)\in\lambda\times S}\\
&=d\circ (e_sf)_{s\in S}\circ 
  (d\circ (\pi_{(\lambda\times S,(i,t))})_{t\in S})_{i\in\lambda}
  \circ (e_t\pi_{(\lambda\times S,(i,t))})_{(i,t)\in\lambda}\circ (\pi_{(\lambda,i)})_{(i,t)\in\lambda\times S}\\
&=d\circ (e_sg)_{s\in S}\circ 
  (d\circ (\pi_{(\lambda\times S,(i,t))})_{t\in S})_{i\in\lambda}
  \circ (e_t\pi_{(\lambda\times S,(i,t))})_{(i,t)\in\lambda}\circ (\pi_{(\lambda,i)})_{(i,t)\in\lambda\times S}\\
&=g.
\end{align*}
Thus, $\nu_\lambda$ is injective.

(Surjectivity)
Suppose $f\in {\rm H}(C_d)$.
By definition of ${\rm H}(C_d)$, $f$ is described as
$f=(f_s/\mc\sim)_{s\in S}$ where $f_s\in C_{\lambda\times S}$
and each $f_s$ satisfies $e_sf_s=f_s$.
Let us define
\[
\tilde{f}:=d\circ (f_s)_{s\in S}
\circ(e_t\pi_{(\lambda,i)})_{(i,t)\in\lambda\times S}.
\]
Then $\nu_\lambda(\tilde{f})$ is calculated as follows:
\begin{align*}
\nu_\lambda(\tilde{f})
&= \left(e_s\tilde{f} \circ
 (d\circ (\pi_{(\lambda\times S,(i,u))})_{u\in S})_{i\in\lambda}/\sim\right)_{s\in S}\\
&=\left(e_s d\circ (f_{s'})_{s'\in S}
\circ(e_t\pi_{(\lambda,i)})_{(i,t)\in\lambda\times S}
\circ (d\circ (\pi_{(\lambda\times S,(i,u))})_{u\in S})_{i\in\lambda}/\mc\sim\right)_{s\in S}\\
&=\left(f_{s}
\circ (e_t\circ \pi_{(\lambda,i)})_{(i,t)\in\lambda\times S}
\circ (d\circ (\pi_{(\lambda\times S,(i,u))})_{u\in S})_{i\in\lambda}/\mc\sim\right)_{s\in S}\\
&=\left(f_{s}
\circ (e_td\circ (\pi_{(\lambda\times S,(i,u))})_{u\in S})_{(i,t)\in\lambda\times S}/\sim\right)_{s\in S}\\
&=\left(f_{s}
\circ (e_t\pi_{(\lambda\times S,(i,t))})_{(i,t)\in\lambda\times S}/\mc\sim\right)_{s\in S}\\
&=(f_{s}/\mc\sim)_{s\in S}\\
&=f.
\end{align*}

(Compatibility)
Let $f\in C_{\lambda_1}$, $g_i\in C_{\lambda_2}$ ($i\in \lambda_1$).
Then 
\begin{align*}
&\nu_{\lambda_2}(f\circ (g_i)_{i\in\lambda_1})\\
=&\left(e_s f\circ(g_i)_{i\in\lambda_1}
  \circ(d\circ(\pi_{(\lambda_2\times S,(j,t))})_{t\in S})_{j\in\lambda_2}/\mc\sim\right)_{s\in S}\\
=&\left(e_s f
  \circ (d\circ(e_u\circ \pi_{(\lambda_1,i)})_{u\in S})_{i\in\lambda_1}
  \circ(g_i)_{i\in\lambda_1}
  \circ(d\circ(\pi_{(\lambda_2\times S,(j,t))})_{t\in S})_{j\in\lambda_2}/\mc\sim\right)_{s\in S}\\
=&\left(e_s f
  \circ (d\circ(\pi_{(\lambda_1\times S,(i,u))})_{u\in S})_{i\in\lambda_1}
  \circ(e_u g_i)_{(i,u)\in\lambda_1\times S}
  \circ(d\circ(\pi_{(\lambda_2\times S,(j,t))})_{t\in S})_{j\in\lambda_2}/\mc\sim\right)_{s\in S}\\
=&\nu_{\lambda_1}(f)\circ (\nu_{\lambda_2}(g_i))_{i\in\lambda_1}.
\end{align*}
\end{proof}

By the isomorphism $\nu$ in this proposition, an ${\rm H}(C_d)$-algebra naturally has 
$C$-algebra structure.
The following isomorphism exists through this identification.
\begin{prop}\label{equiv-single}
Let $C$ be a $<\mc\kappa$-ary clone with a diagonal pair $(d,(e_s)_{s\in S})$.
Let $A$ be a $C$-algebra. Then $\nu_A:A\ni a\mapsto (e_s(a))_{s\in S}\in {\rm H}(C_d)$
is an isomorphism of $C$-algebras.
Furthermore, $(\nu_A)_{A\in\V(C)}$ is a natural transformation
from $\id_{\Cat(\V(C))}$ to the functor $[A\mapsto {\rm H}(A_d)]$.
\end{prop}
\begin{proof}
By definition of diagonal pair, $(a_s)_{s\in S}\mapsto d(a_s)_{s\in S}$ is 
the inverse mapping of $\nu_A$. Thus $\nu_A$ is bijective.

We prove $\nu_A$ is a homomorphism of $C$-algebras.
Let $f\in C_\lambda$, $\lambda<\kappa$ and $(a_i)_{i\in\lambda}\in A^\lambda$.
Then
\begin{align*}
\nu_\lambda(f)(\nu_A(a_i))_{i\in\lambda}
&=(e_sf(d(e_t(a_i))_{t\in S})_{i\in\lambda})_{s\in S}\\
&=(e_sf(a_i)_{i\in\lambda})_{s\in S}\\
&=\nu_A(f(a_i)_{i\in\lambda}).
\end{align*}
This equation is nothing but the property we should prove.

Next, we prove naturalness of $(\nu_A)_{A\in\V(C)}$.
Let $A,B\in\V(C)$ and $\ph:A\ra B$ be a homomorphism.
We should the diagram
\[\xymatrix{
A\ar[r]^{\nu_A} \ar[d]_\ph& {\rm H}(A_d)\ar[d]^{{\rm H}(\ph_d)}\\
B\ar[r]^{\nu_B} & {\rm H}(B_d)
}\]
commute. For $a\in A$,
\begin{align*}
{\rm H}(\ph_d)(\nu_A(a))
&={\rm H}(\ph_d)(e_s(a))_{s\in S}\\
&=(\ph(e_s(a))_{s\in S}\\
&=(e_s(\ph(a))_{s\in S}\\
&=\nu_B(\ph(a))
\end{align*}
holds. It means $(\nu_A)_{A\in\V(C)}$ is a natural transformation.
\end{proof}

We complete to prove that
the variety corresponding to the clone ${\rm H}(C_d)$ is definitionally
equivalent to the variety corresponding to $C$ via the natural equivalence $(\nu_A)_{A\in\V}$.
Next, we prove a natural equivalence of the converse direction.

\begin{prop}\label{clone-many}
Let $M$ be an $S$-sorted $<\mc\kappa$-ary pure clone.
Let $d\in {\rm H}_S(M)$ and $e_s\in {\rm H}_1(M)$ ($s\in S$) be terms that
satisfy the following properties:
\begin{itemize}
\item
The term $d$ is defined as $d:=(\pi_{(S\times S,(s,s))})_{s\in S}$.
\item
Each $e_s$ can be written as in the form
$e_s=(e_{s,t}\pi_{(S,\id,s)})_{t\in S}$, where $e_{s,t}\in M_{s\ra t}$
and satisfies \mbox{$e_{s,s}=\pi_{(\{s\},\id,s)}$}.
\end{itemize}
Then $(d,(e_s)_{s\in S})$ is a diagonal pair of ${\rm H}(M)$ and
$({\rm H}(M))_d$ is isomorphic to $M$ via an isomorphism
\[
\mu_{(\lambda,v,s)}:M_{(\lambda,v,s)}\ni
f \mapsto (e_{s,t})_{t\in S}\circ f\circ 
(
\pi_{(\lambda\times S,p_2,(i,v(i)))})_{i\in\lambda}/\mc\sim
\in ({\rm H}(M))_{d,(\lambda,v,s)},
\]
where $\sim$ is the equivalence
\begin{align*}
&(f_s)_{s\in S}\sim (g_s)_{s\in S}\\
:\Longleftrightarrow&
  (f_s)_{s\in S}\circ (e_{v(i)}\pi_{(\lambda,i)}^{{\rm H}(M)})_{i\in\lambda}
 =(g_s)_{s\in S}\circ (e_{v(i)}\pi_{(\lambda,i)}^{{\rm H}(M)})_{i\in\lambda}\\
\Longleftrightarrow&
 \forall s\in S;
  f_s\circ (e_{v(i),t}\pi_{(\lambda\times S,p_2,(i,v(i)))})_{(i,t)\in\lambda\times S}
 =g_s\circ (e_{v(i),t}\pi_{(\lambda\times S,p_2,(i,v(i)))})_{(i,t)\in\lambda\times S},
\end{align*}
where $\pi^{{\rm H}(M)}$ denotes the projection constant of ${\rm H}(M)$.
\end{prop}
\begin{proof}
First, we prove $(d,(e_s)_{s\in S})$ is a diagonal pair of ${\rm H}(M)$.

1. $e_s\circ d=e_s\circ \pi_{(S,s)}$ is proved as follows:
\[
e_sd
= e_s(\pi_{(S\times S,p_2,(t,t))}^M)_{t\in S}
=e_s(\pi_{(S\times S,p_2,(s,t))}^M)_{t\in S}
= e_s\circ \pi_{(S,s)}^{{\rm H}(M)}.
\]
Here and through the proof of this proposition,
$\pi^M$ and $\pi^{{\rm H}(M)}$ denote projection constants
of $M$ and ${\rm H}(M)$ respectively.

2. $d\circ(e_s\circ\pi_{(S,s)})=d$ is proved as follows:
\begin{align*}
d(e_s\circ \pi_{(S,s)}^{{\rm H}(M)})_{s\in S}
&= (\pi_{(S\times S,p_2,(s,s))})_{s\in S}
  \circ (e_{s,t}\pi_{(S,\id,s)}^M\circ
       (\pi_{(S\times S,p_2,(s,u))})_{u\in S})_{(s,t)\in S\times S}\\
&=(e_{s,s}\pi_{(S,\id,s)}^M\circ (\pi_{(S\times S,p_2,(s,u))}^M)_{u\in S})_{s\in S} \\
&=(\pi_{(S\times S,p_2,(s,s))}^M)_{s\in S}\\
&=d.
\end{align*}

3. $d\circ(\pi_{(S,0)})_{s\in S}=\pi_{(S,0)}$ is proved as follows:
\begin{align*}
d\circ (\pi_{(S,0)}^{{\rm H}(M)})_{s\in S}
&=(\pi_{(S\times S,p_2,(s,s))}^M)_{s\in S}
  \circ (\pi_{(S\times S,p_2,(0,t))}^M)_{(s,t)\in S\times S}\\
&=(\pi_{(S\times S,p_2,(0,s))}^M)_{s\in S}\\
&=\pi_{(S,0)}^{{\rm H}(M)}.
\end{align*}
We complete the proof that $(d,(e_s)_{s\in S})$ is a diagonal pair of ${\rm H}(M)$.

(Injectivity of $\mu_{(\lambda,v,s)}$)
Let $f,g\in M_{(\lambda,v,s)}$ and assume $\mu_{(\lambda,v,s)}(f)=\mu_{(\lambda,v,s)}(g)$,
namely,
\begin{align*}
&e_{s,t}f\circ(\pi_{(\lambda\times S,p_2,(i,v(i))}^M))_{i\in\lambda}
 \circ (e_{v(i),u}\pi_{(\lambda\times S,p_2,(i,v(i))})_{(i,u)\in\lambda\times S}\\
=&e_{s,t}g\circ(\pi_{(\lambda\times S,p_2,(i,v(i))}^M))_{i\in\lambda}
 \circ (e_{v(i),u}\pi_{(\lambda\times S,p_2,(i,v(i))})_{(i,u)\in\lambda\times S}
\end{align*}
hold for all $t\in S$.
Notice that the following two equations hold: The first equation is
\begin{align*}
 &\pi_{(\lambda,v,v(i))}^M\\
=&e_{v(i),v(i)}\circ \pi_{(\lambda,v,v(i))}^M \\
=&\pi_{(\lambda\times S,p_2,(i,v(i)))}^M
 \circ (e_{v(j),u}\circ \pi_{(\lambda,v,v(j))}^M)_{(j,u)\in\lambda\times S}
\end{align*}
The second equation is
\[
\id_s:=\pi_{(\{s\},\id_{\{s\}},s)}^M=e_{s,s}\pi_{(S,\id_S,s)}^M (e_{s,t})_{t\in S}.
\]
By these equations injectivity of $\mu_{(\lambda,v,s)}$ is proved as follows:
\begin{align*}
f
&=\id_s\circ f\circ (\pi_{(\lambda,v,v(i))}^M)_{i\in\lambda}\\
&=e_{s,s}\pi_{(S,\id,s)}^M (e_{s,t})_{t\in S}\circ f
 (\pi_{(\lambda\times S,p_2,(i,v(i)))}^M)_{i\in\lambda}
 \circ (e_{v(j),u}\circ \pi_{(\lambda,v,v(j))}^M)_{(j,u)\in \lambda\times S}\\
&=e_{s,s}\pi_{(S,\id,s)}^M (e_{s,t})_{t\in S}\circ g
 (\pi_{(\lambda\times S,p_2,(i,v(i)))}^M)_{i\in\lambda}
 \circ (e_{v(j),u}\circ \pi_{(\lambda,v,v(j))}^M)_{(j,u)\in \lambda\times S}\\
&=g.
\end{align*}
(Surjectivity of $\mu_{(\lambda,v,s)}$)
Let $f/\mc\sim\in ({\rm H}(M))_{d,(\lambda,v,s)}$. Define $\tilde{f}$ by
$\tilde{f}=\pi_{(S,\id,s)}^M\circ f\circ
(e_{v(i),u}\circ\pi_{(\lambda,v,i)})_{(i,u)\in \lambda\times S}$.
Then the following formula holds.
\begin{align*}
\mu_{(\lambda,v,s)}(\tilde{f})
&=\left((e_{s,t})_{t\in S}\circ \tilde{f}\circ
 ( \pi_{(\lambda\times S,p_2,(i,v(i)))}^M)_{i\in\lambda}\right)/\mc\sim\\
&=\left((e_{s,t})_{t\in S}\circ \pi_{(S,\id,s)}^M\circ f\circ
 (e_{v(i),u}\circ \pi_{(\lambda,v,i)}^M)_{(i,u)\in \lambda\times S}\circ
 (\pi_{(\lambda\times S,p_2,(i,v(i)))}^M)_{i\in\lambda}\right)/\mc\sim\\
&=\left(e_s\circ f\circ
 (e_{v(i),u}\circ \pi_{(\lambda,v,i)}^M)_{(i,u)\in \lambda\times S}\circ
 (\pi_{(\lambda\times S,p_2,(i,v(i)))}^M)_{i\in\lambda}\right)/\mc\sim\\
&=\left(f\circ
 (e_{v(i),u}\circ \pi_{(\lambda\times S,p_2,(i,v(i)))}^M)_{(i,u)\in\lambda\times S}\right)/\mc\sim\\
&= f/\mc\sim.
\end{align*}
(Compatibility of $\mu$)
Let $f\in M_{(\lambda_1,v_1,s)}$ and
$g_i\in M_{(\lambda_2,v_2,v_1(i))}$ for $i\in\lambda_1$.
Then
\begin{align*}
&\mu_{(\lambda_2,v_2,s)}(f\circ (g_i)_{i\in\lambda_1})\\
=&\left((e_{s,t})_{t\in S}\circ f\circ (g_i)_{i\in\lambda_1}
 \circ (\pi_{(\lambda_2\times S,p_2,(j,v_2(j))}^M)_{j\in\lambda_2}\right)/\mc\sim\\
=&\left((e_{s,t})_{t\in S}\circ f\circ (e_{v_1(i),v_1(i)}g_i)_{i\in\lambda_1}
 \circ (\pi_{(\lambda_2\times S,p_2,(j,v_2(j))}^M)_{j\in\lambda_2}\right)/\mc\sim\\
=&\left((e_{s,t})_{t\in S}\circ f
 \circ (\pi_{(\lambda_1\times S,p_2,(i,v_1(i))}^M)_{i\in\lambda_1}
 \circ (e_{v_1(i),u})_{(i,u)\in\lambda_1\times S}\circ (g_i)_{i\in\lambda_1}
 \circ (\pi_{(\lambda_2\times S,p_2,(j,v_2(j))}^M)_{j\in\lambda_2}\right)/\mc\sim\\
=&\mu_{(\lambda_1,v_1,s)}(f)\circ (\mu_{(\lambda_2,v_2,v_1(i))}(g_i))_{i\in\lambda_1}.
\end{align*}
We complete the proof of the proposition.
\end{proof}

By this proposition, an $({\rm H}(M))_d$-algebra is identified with an
$M$-algebra. Under this identification, the following holds.
\begin{prop}\label{equiv-many}
Let $M,d,e_s,e_{s,t}$ be as in the previous proposition,
$A=(A_s)_{s\in S}$ be an $M$-algebra.
Then $({\rm H}(A))_d$ is isomorphic to $A$ via an isomorphism $\mu_A$, where
\[
\mu_{A,s}:A_s\ni a\mapsto (e_{s,t}(a))_{t\in S}\in ({\rm H}(A))_{d,s}.
\]
Moreover, $(\mu_A)_{A\in\V(M)}$ is a natural transformation
$\id_{\Cat(\V(M))}$ to $[A\mapsto ({\rm H}(A))_d]$.
\end{prop}
\begin{proof}
Let $f\in M_{(\lambda,v,s)}$ and $a_i\in A_{v(i)}$ for $i\in\lambda$.
Then
\begin{align*}
\mu_{(\lambda,v,s)}(f)(\mu_{A,v(i)}(a_i))_{i\in\lambda}
&=(e_{s,t})_{t\in S}\circ f\circ (\pi_{(\lambda\times S,p_2,(i,v(i)))})_{i\in\lambda}
(e_{v(i),u}(a_i))_{(i,u)\in\lambda\times S}\\
&=(e_{s,t})_{t\in S}\circ f(e_{v(i),v(i)}(a_i))_{i\in\lambda}\\
&=(e_{s,t})_{t\in S}\circ f(a_i)_{i\in\lambda}\\
&=\mu_{A,s}(f(a_i)_{i\in\lambda}).
\end{align*}
This equation is nothing but the proposition stated.

We prove $(\mu_A)_{A\in\V(M)}$ is a natural transformation.
Let $A=(A_s)_{s\in S},B=(B_s)_{s\in S}\in \V(M)$ and $\ph=(\ph_s)_{s\in S}$
be a homomorphism $A\ra B$.
We prove should the diagram
\[\xymatrix{
A\ar[r]^{\mu_A} \ar[d]_\ph& ({\rm H}(A))_d\ar[d]^{({\rm H}(\ph))_d}\\
B\ar[r]^{\mu_B} & ({\rm H}(B))_d
}\]
commute. For $s\in S$ and $a\in A_s$, the equation
\begin{align*}
\mu_{B,s}(\ph_s(a))
&=(e_{s,t}(\ph_s(a))_{t\in S}\\
&=(\ph_t(e_{s,t}(a)))_{t\in S}\\
&={\rm H}(\ph)(e_{s,t}(a))_{t\in S}\\
&=({\rm H}(\ph))_{d,s}(\mu_{A,s}(a))
\end{align*}
holds.
This is nothing but $\mu_B\circ \ph=({\rm H}(\ph))_d\circ \mu_A$, thus
it is proved that $(\mu_A)_{A\in\V(M)}$ is a natural transformation.
\end{proof}

By Proposition \ref{clone-single} and \ref{clone-many},
there is natural correspondence between many-sorted pure clones
and single-sorted clones with diagonal operations
(that can be extended to diagonal pair).
Proposition \ref{many-single},\ref{single-many},\ref{equiv-single}
and \ref{equiv-many}, many-sorted varieties are essentially equivalent to single-sorted varieties that have diagonal pairs.
We state these facts explicitly.

\begin{thm}
Let $\kappa$ be infinite cardinal and $S$ be a cardinal that $0<S<\kappa$.
Then $(C,d)\mapsto C_d$ is one to one correspondence between
(isomorphism classes of) single-sorted $<\mc\kappa$-ary clones with $S$-ary diagonal operations
and (isomorphism classes of) $S$-sorted $<\mc\kappa$-ary pure clones.
The converse correspondence is given by
$M\mapsto ({\rm H}(M),(\pi_{(S\times S,p_2,(s,s))})_{s\in S})$.
\end{thm}
Here, a $<\mc\kappa$-ary clone with an $S$-ary diagonal operation means
a tuple $(C,d)$ that $C$ is a $<\mc\kappa$-ary clone, $d\in C_S$ and
there exists $(e_s)_{s\in S}$ such that $(d,(e_s)_{s\in S})$ is
a diagonal pair of $C$.
Clones with diagonal operations $(C,d),(C',d')$ are said to be isomorphic
if there is an clone isomorphism $\ph:C\ra C'$ such that $\ph(d)=d'$.

\begin{thm}\label{main-theorem-categorical-equivalence}
Let $\kappa$ be infinite cardinal, $S$ be a cardinal that $0<S<\kappa$,
$M$ be an $S$-sorted $<\mc\kappa$-ary clone.
Then there exists a categorical equivalence $\nu:\V(M)\ra \V({\rm H}(M))$
such that the following diagram commute.
\[\xymatrix{
\V(M)\ar[r]^\nu \ar[d]_F& \V({\rm H}(M))\ar[d]^G\\
\Set^S \ar[r]^\ph & {\cal D}_S
}\]
The vertical arrow $F$ is the forgetful functors attaining the underlying family of sets. $G$ is the forgetful functor determined by the clone
homomorphism $d\in C_{{\cal D}_S}$ maps to $(\pi_{(S\times S,p_2,(s,s))}^M)_{s\in S}$.
The lower horizontal arrow $\ph$ is the functor described in 
Proposition \ref{pureset-and-sets}
\end{thm}

\section{Applications} \label{s-applications}

By the correspondence between many-sorted algebras and single-sorted algebras,
many results on single-sorted algebras are generalized to many-sorted algebras.
At the end of this paper, we show few such examples.

\subsection{Variety theorem} \label{s-variety-theorem}

The first example is a many-sorted version of Birkhoff's characterization of varieties.
(See also \cite{ARV}. Another version of many-sorted variety theorem.)
We start from compatibility of homogenization and generating operators
homomorphic images $\mathbf{H}$, subalgebras $\mathbf{S}$ and direct products $\mathbf{P}$
of variety.
\begin{df}
Let $A=(A_s)_{s\in S}$ be an $S$-sorted algebra.
\begin{enumerate}
\item
A subalgebra of $A$ is an $S$-sorted family $B=(B_s)_{s\in S}$ that satisfies
the following conditions:
\begin{itemize}
\item
$B_s\subset A_s$ for all $s\in S$.
\item
For $(\lambda,v)$-ary $s$-valued term operation $f$ of $A$ and
$(b_i)_{i\in\lambda}\in \prod_{i\in\lambda}B_{v(i)}$,
$f(b_i)_{i\in \lambda}\in B_s$ holds.
\end{itemize}
The set of all subalgebras of $A$ is denoted by $\Sub(A)$.
\item
A congruence of $A$ is a subalgebra $(\theta_s)_{s\in S}$ of $A^2=(A_s{}^2)_{s\in S}$ such that
$\theta_s$ is an equivalence relation on $A_s$ for each $s\in S$.
The set of all congruences of $A$ is denoted by $\Con(A)$.
\end{enumerate}
\end{df}

\begin{prop}\label{sub-con-corresopndence}
Let $M$ be an $S$-sorted clone, $A$ be an $M$-algebra.
\begin{enumerate}
\item
If $B=(B_s)_{s\in S}\in \Sub(A)$, then ${\rm H}(B):=\prod_{s\in S}B_s\in \Sub({\rm H}(A))$.
Conversely, Any subalgebra of ${\rm H}(A)$ is described as in the form
${\rm H}(B)$ by $B\in\Sub(A)$.
\item
If $\theta=(\theta_s)_{s\in S}\in\Con(A)$, then
\[
{\rm H}(\theta):=\prod_{s\in S}\theta_s
:=\{((a_s)_{s\in S},(b_s)_{s\in S})\in ({\rm H}(A))^2\mid
 \forall s\in S;(a_s,b_s)\in\theta_s\} \in\Con({\rm H}(A)).
\]
Conversely, Any congruence of ${\rm H}(A)$ is described as in the form
${\rm H}(\theta)$ by a congruence $\theta\in\Con(A)$.
\item
If $\theta\in \Con(A)$, then ${\rm H}(A/\theta)\iso {\rm H}(A)/{\rm H}(\theta)$.
\item
If $(A_i)_{i\in I}$ is a family of $M$-algebras, then
\[
\prod_{i\in I}{\rm H}(A_i)\iso {\rm H}\left(\prod_{i\in I} A_i\right).
\]
\end{enumerate}
\end{prop}
\begin{proof}
We only show that any subalgebra of ${\rm H}(A)$ is in the form ${\rm H}(B)$ by
a subalgebra $B$ of $A$.
Other parts are easily verified or proved similarly.

If ${\rm H}(A)= \emptyset$, then the assertion is trivial.
Thus, we assume $A_s\neq\emptyset$ for all $s\in S$.
Let $B'$ be a subalgebra of ${\rm H}(A)$. Define
\[
B_s:=\{b\in A_s\mid\exists (a_t)_{t\in S\sm\{s\}}\in\prod_{t\in S\sm\{s\}}A_t; (b,(a_t)_{t\in S\sm\{s\}})\in B'\}.
\]
Clearly, $B'\subset \prod_{s\in S}B_s$ holds. We show $\prod_{s\in S}B_s\subset B'$.
Let $(b_s)_{s\in S}\in \prod_{s\in S}B_s$. By definition of $B_s$,
there exist $(c_{st})_{t\in S}\in B'$ such that $c_{ss}=b_s$.
Let $d=(\pi_{(S\times S,p_2,(s,s))})_{s\in S}$ be the diagonal term. Then
\[
B'\ni d((c_{st})_{t\in s})_{s\in S}=(c_{ss})_{s\in S}=(b_s)_{s\in s}.
\qedhere
\]
\end{proof}

\begin{rem}\label{subalgebra-of-pure-algebra}
If $M$ is a pure clone, then the underlying family of sets of any subalgebra of an $M$-algebra $A$ is pure. In such case, the correspondence $B\mapsto {\rm H}(B)$
between $\Sub(A)$ and $\Sub({\rm H}(A))$ is bijective. On the other hand,
this correspondence may not be bijective if $M$ is not pure.
For a simple example, $S=2=\{0,1\}$, $M$ be the trivial $2$-sorted clone
(that is, the $2$-sorted clone only consists of projection constants)
and $A=(A_i)_{i\in 2}=(A_0,A_1)$, then the both $(\emptyset,A_1)$ and $(A_0,\emptyset)$
of members of $\Sub(A)$ correspond to $\emptyset=\emptyset\times A_1=A_0\times\emptyset$.
\end{rem}

\begin{thm}\label{variety-theorem}
Let $M$ be an $S$-sorted pure clone, $\mathcal{K}$ be a class of $M$-algebras.
Then the variety $\V(\mathcal{K})$ generated by $\mathcal{K}$ coincides with
$\mathbf{HSP}(\mathcal{K})$.
\end{thm}
\begin{proof}
We show
\begin{equation}\label{varieties-through-homogenization}
{\rm H}(\V(\K))=\V({\rm H}(\K))=\mathbf{HSP}({\rm H}(\K))={\rm H}(\mathbf{HSP}(\K)).
\end{equation}
The second equality is Birkhoff's variety theorem for classes of single-sorted algebras.
The last equality follows from Proposition \ref{sub-con-corresopndence}.

The first equality is proved as follows. 
For $f,g\in {\rm H}_\lambda(M)$, where $f=(f_s)_{s\in S}$, $g=(g_s)_{s\in S}$,
the following equivalence holds:
\begin{align*}
{\rm H}(\V(\K))\models f=g
\Longleftrightarrow& \forall s\in S;\ \V(\K)\models f_s=g_s\\
\Longleftrightarrow& \forall s\in S;\ \K\models f_s=g_s\\
\Longleftrightarrow& {\rm H}(\K)\models f=g\\
\Longleftrightarrow& \V({\rm H}(\K))\models f=g.
\end{align*}
It means the equational theories defining $\V({\rm H}(\K))$ and ${\rm H}(\V(\K))$
coincide with each other.
Therefore, ${\rm H}(\V(\K))=\V({\rm H}(\K))$ holds.
By (\ref{varieties-through-homogenization}) and Theorem \ref{main-theorem-categorical-equivalence},
$\V(\K)=\mathbf{HSP}(\K)$ holds.
\end{proof}

\begin{rem}
As noted in preliminary section, the above clone-based
formulation includes type-based formulation
as the case that $M$ is a free clone.
Thus, Theorem \ref{variety-theorem} above is essentially equivalent
to usual formulation of variety theorem.
\end{rem}

\subsection{Relational clone and categorical equivalence}
\label{s-relational-clone-and-categorical-equivalence}

Denecke and L\"{u}ders proved in \cite{DL} that two finite single-sorted algebras are
categorically equivalent if and only if their relational clones are isomorphic.
In this subsection, we introduce the notion of a relational clone of a many-sorted algebra, and we show that a relational clone of an algebra
such that the clone of its term operations is pure is isomorphic to the relational
clone of its homogenization.
As a consequence of this fact, we obtain a characterization of categorical equivalence
by isomorphism relation of corresponding relational clones for many-sorted algebras.

\begin{df}
Let $A$ be an $S$-sorted $<\mc\kappa$-ary algebra, $\kappa'$ be a cardinal.
A $<\mc\kappa'$-ary relational clone of $A$ is the family $\Inv_{<\mc\kappa'}(A)=(\Inv_\mu(A))_{\mu<\kappa'}$,
where $\Inv_\mu(A)=\{{\rm H}(B)\mid B\in\Sub(A)\}$,
namely, $\Inv_\mu(A)$ is the set of all subsets $r$ of
$\left(\prod_{s\in S}A_s\right)^\mu$ that satisfy
\begin{align*}
&f_s\in \Clo_{(\lambda\times S,p_2,s)}(A),(a_{i,j,t})_{(t,j)\in S\times\mu}\in r\ (\text{for }i\in\lambda)\\
\Longrightarrow& (f_s(a_{i,j,t})_{(i,t)\in\lambda\times S})_{(s,j)\in S\times\mu}\in r.
\end{align*}
Here, $\Clo_{(\lambda,v,s)}(A)$ is the set of all $(\lambda,v)$-ary $s$-valued term operations of $A$.
A member of $\Inv_\mu(A)$ is said an $\mu$-ary invariant set of $A$.
\end{df}

\begin{rem}
We can consider the set of all closed subsets for component-wise operations.
Here, a set
$r\subset \prod_{j\in\mu}A_{v(j)}$ ($\mu$ is a cardinal, $v:\mu\ra S$) is said
closed if the following condition holds:
For $\mu<\kappa'$ and $v:\mu\ra S$
\begin{align*}
&f_j\in \Clo_{(\lambda,[i\mapsto v(j)],v(j))}(A),(a_{i,j})_{j\in\mu}\in r\ (\text{for }i\in\lambda)\\
\Longrightarrow& (f_j(a_{i,j})_{i\in\lambda})_{j\in\mu}\in r.
\end{align*}
However, this condition considers only on operations related at most one sort.
Thus, it should be said a concept of the non-indexed product
of an one-sorted algebras, rather than a concept of a many-sorted algebra.
\end{rem}

Next, we define the structure of relational clone and define the notion of isomorphism.
As the single-sorted case, invariant sets are closed under primitive-positive definitions. It is easily verified as the same way as the single-sorted case.
\begin{prop}
Let $A$ be an $S$-sorted $<\mc\kappa$-ary algebra and $r_k\in \Inv_{\mu_k}(A)$ for $k\in K$.
Assume a relation $r\subset \left(\prod_{s\in S}A_s\right)^{\mu}$ is defined by the following form
\[
r=\left\{(a_{s,j'})_{(s,j')\in S\times\mu}\in \left(\prod_{s\in S}A_s\right)^{\mu}\mid 
  \exists (a_{s,\tilde{j}})_{(s,\tilde{j})\in S\times \nu}\in \left(\prod_{s\in S}A_s\right)^\nu;
  \bigwedge_{u\in U} (a_{s,f(u,j)})_{(s,j)\in \mu_{g(u)}}\in r_{g(u)}\right\},
\]
where $g:U\ra K$ and $f:\bigcup_{u\in U}\{u\}\times \mu_{g(u)}\ra \mu\amalg \nu$.
Then $r\in \Inv_\mu(A)$ holds.
Here, $\mu\amalg\nu$ denotes the disjoint union of $\mu$ and $\nu$.
\end{prop}

\begin{df}
Let $\kappa'$, $S_A$ and $S_B$ be cardinals, $A$ and $B$ be $S_A$-sorted, $S_B$-sorted algebra respectively.
A $<\mc\kappa'$-ary relational clone homomorphism $\Inv_{<\mc\kappa'}(A)\ra \Inv_{<\mc\kappa'}(B)$
is a family $(\ph_\mu)$ of mappings that $\ph_\mu:\Inv_\mu(A)\ra \Inv_\mu(B)$ and
preserves primitive-positive definitions, namely,
if $r_k\in \Inv_{\mu_k}(A)$ and
\[
r=\left\{(a_{s,j'})_{(s,j')\in S\times\mu}\in \left(\prod_{s\in S}A_s\right)^\mu\mid 
  \exists (a_{s,\tilde{j}})_{(s,\tilde{j})\in S\times\nu}\in\left(\prod_{s\in S}A_s\right)^\nu; 
  \bigwedge_{u\in U} (a_{s,f(u,j)})_{(s,j)\in \mu_{g(u)}}\in r_{g(u)}\right\},
\]
then
\begin{multline*}
\ph_\mu(r)=\Bigg\{(b_{s,j'})_{(s,j')\in S\times\mu}\in
 \left(\prod_{s\in S}B_s\right)^\mu\mid \\
  \exists (b_{s,\tilde{j}})_{(s,\tilde{j})\in S\times\nu}\left(\prod_{s\in S}B_s\right)^\nu; 
  \bigwedge_{u\in U} (b_{s,f(u,j)})_{(s,j)\in \mu_{g(u)}}\in \ph_{\mu_{g(u)}}(r_{g(u)})\Bigg\}.
\end{multline*}
\end{df}

The next is the main theorem about relational clones of many-sorted algebras.
However, the proof is extremely easy;
it directly follows from the definition of homogenization and relational clone.
\begin{thm}\label{many-single-relationa-clone}
Let $M$ be an $S$-sorted $<\mc\kappa$-ary pure clone,
$A$ be an $M$-algebra, and $\kappa'$ be an arbitrary cardinal.
Then the $<\mc\kappa'$-ary relational clone $\Inv_{<\kappa'}(A)$ is isomorphic to
$\Inv_{<\kappa'}({\rm H}(A))$.
\end{thm}
\begin{proof}
By the definition of homogenization and relational clone,
\[
\Inv_\mu(A)\ni r\mapsto 
 \{((a_{j,s})_{s\in S})_{j\in\mu}\in ({\rm H}(A))^\mu\mid(a_{j,s})_{(j,s)\in\mu\times S}\in r\} \in \Inv_{\mu}({\rm H}(A))
\]
is an isomorphism of relational clone.
Note that the bijectivity of this correspondence follows from pureness of $M$
(Remark \ref{subalgebra-of-pure-algebra}).

\end{proof}

Using the result on single-sorted algebras, we obtain a characterization
of categorical equivalence of many-sorted algebras.
\begin{cor}
Let $S_k$ be a non-zero cardinal, $\kappa$ be an infinite cardinal that $S_k<\kappa$ for $k=1,2$.
Let $A_k$ be an $S_k$-sorted $<\mc\kappa$-ary algebra that the clones of term operations are pure for $k=1,2$.
Let $\kappa'$ be an infinite cardinal that satisfies
\[
\lambda<\kappa\ \Longrightarrow |A_1|^\lambda, |A_2|^\lambda <\kappa',
\]
where $|A_k|$ is the product cardinal $\prod_{s\in S_k} |A_{k,s}|$.
Then the following assertions are equivalent.
\begin{enumerate}
\item
There exists a categorical equivalence $\V(A_1)\ra\V(A_2)$ that maps $A_1$ to $A_2$.
\item
The relational clones $\Inv_{<\mc\kappa'}(A_1)$ is isomorphic to $\Inv_{<\mc\kappa'}(A_2)$.
\end{enumerate}
\end{cor}
\begin{proof}
This corollary follows from 
Theorem \ref{main-theorem-categorical-equivalence},
Theorem \ref{many-single-relationa-clone} and 
the following theorem.
\end{proof}

\begin{thm}[\cite{Izainf}]
Let $\kappa$ be an infinite cardinal and
$A_k$ be a single-sorted $<\mc\kappa$-ary algebra for $k=1,2$.
Let $\kappa'$ be an infinite cardinal that satisfies
$\lambda<\kappa\ \Longrightarrow |A_1|^\lambda,|A_2|^\lambda<\kappa'$.
Then the following assertions are equivalent.
\begin{enumerate}
\item
There exists a categorical equivalence $\V(A_1)\ra\V(A_2)$ that maps $A_1$ to $A_2$.
\item
The relational clones $\Inv_{<\mc\kappa'}(A_1)$ is isomorphic to $\Inv_{<\mc\kappa'}(A_2)$.
\end{enumerate}
\end{thm}

\subsection{Mal'cev type characterization} \label{s-malcev-type-characterization}

As a special case of invariant relations, there are bijective correspondences
between subalgebras or congruences of a many-sorted algebra and subalgebras or congruences
of its homogenization.
This fact implies various many-sorted generalization of results on single-sorted algebras.
As an example, we show characterization theorems of congruence properties by
term existence conditions.

The first example is Mal'cev's characterization of congruence permutability.
To make precise, we start from a definition.
\begin{df}
Let $A$ be an $S$-sorted algebra.
\begin{enumerate}
\item
Let $\theta_s,\eta_s\subset A_s{}^2$.
A relational product $\theta\circ\eta$ of $\theta=(\theta_s)_{s\in S}$ and $\eta=(\eta_s)_{s\in S}$
is defined by sort-wise relational product $(\theta_s\circ\eta_s)_{s\in S}$. 
\item
An algebra $A$ is said congruence  permutable if $\theta\circ\eta=\eta\circ\theta$
for all $\theta,\eta\in\Con(A)$.
\item
A class of algebras $\K$ is said congruence permutable if all members of
$\K$ are congruence permutable.
\end{enumerate}
\end{df}

\begin{thm}\label{malcev}
Let $M$ be an $S$-sorted pure clone.
Then the following conditions are equivalent.
\begin{enumerate}
\item
$\V(M)$ is congruence permutable.
\item
$\V({\rm H}(M))$ is congruence permutable.
\item
There exists $p\in {\rm H}_3(M)$, called Mal'cev term, that satisfies
\[
p(\pi_{(3,0)},\pi_{(3,0)},\pi_{(3,1)})=\pi_{(3,1)}
,p(\pi_{(3,0)},\pi_{(3,1)},\pi_{(3,1)})=\pi_{(3,0)},
\]
where $3$ is a three elements set $\{0,1,2\}$.
\item
For each $s\in S$, there exists $p_s\in M_{(3\times S,p_2,s)}$ such that
\begin{align*}
p_s((\pi_{(3\times S,p_2,(0,t))})_{t\in S},(\pi_{(3\times S,p_2,(0,t))})_{t\in S},(\pi_{(3\times S,p_2,(1,t))})_{t\in S})
=\pi_{(3\times S,p_2,(1,t))},\\
p_s((\pi_{(3\times S,p_2,(0,t))})_{t\in S},(\pi_{(3\times S,p_2,(1,t))})_{t\in S},(\pi_{(3\times S,p_2,(1,t))})_{t\in S})
=\pi_{(3\times S,p_2,(0,t))}.
\end{align*}
\item
For each $s\in S$, there exists $p_s\in M_{3s\ra s}$ such that
\begin{align*}
p_s(\pi_{(3s\ra s,0)},\pi_{(3s\ra s,0)},\pi_{(3s\ra s,1)})
=\pi_{(3s\ra s,1)},\\
p_s(\pi_{(3s\ra s,0)},\pi_{(3s\ra s,1)},\pi_{(3s\ra s,1)})
=\pi_{(3s\ra s,0)}.
\end{align*}
\end{enumerate}
\end{thm}
\begin{proof}
1 $\Leftrightarrow$ 2 is an easy consequence of Proposition \ref{sub-con-corresopndence}.

2 $\Leftrightarrow$ 3 is well known Mal'cev's theorem (See e.g.{\cite[Theorem 12.2]{BS}}).

3 $\Leftrightarrow$ 4 follows from the definition of homogenization.

5 $\Rightarrow$ 4 is trivial.

1 $\Rightarrow$ 5. Let $F$ be a rank $(3\delta_{st})_{t\in S}$ free algebra,
namely, there are three elements $a,b,c$ of sort $s$ such that $\{a,b,c\}$ freely
generates $F$. Let $\theta$ and $\eta$ be congruences of $F$ generated by
$(a,b)$ and $(b,c)$ respectively. Then $(a,c)\in\theta\circ\eta=\eta\circ\theta$,
namely, there is $x\in F_s$ such that $(a,x)\in \eta$, $(x,c)\in\theta$.
Let $p_s$ be the term corresponding to $x$, then equations stated in Condition 5 hold.
\end{proof}

\begin{cor}
Let $C$ be a $<\mc\kappa$-ary single-sorted clone,
$(d,(e_s)_{s\in S})$ be a diagonal pair of $C$. Then the following conditions are equivalent.
\begin{enumerate}
\item
The variety $\V(C)$ is congruence permutable.
\item
The variety $\V(e_s(C))$ is congruence permutable for all $s\in S$.
\end{enumerate}
\end{cor}
\begin{proof}
Condition 1 of this corollary is equivalent to that $\V_d$ satisfy the condition 1 of the previous theorem.
Condition 2 of this corollary is equivalent to that $\V_d$ satisfy the condition 5 of the previous theorem.
Thus, these conditions are equivalent.
\end{proof}

Similar results hold about congruence distributivity, modularity, etc.. 
However, many results on Mal'cev condition that does not fixed the number (or the form of equation)
depend on finiteness of arity.
In such case, the corresponding generalized results are limited.
Moreover, we can construct some counter examples of infinitary or infinitely many-sorted classes
by using the correspondence of many-sorted and single-sorted algebras.

We only show a generalization of characterization of congruence distributivity.
\begin{thm}
Let $S$ be a finite set, $M$ be an $S$-sorted finitary pure clone.
Then the following conditions are equivalent:
\begin{enumerate}
\item
$\V(M)$ is congruence distributive.
\item
$\V({\rm H}(M))$ is congruence distributive.
\item
There are an integer $n\geq 0$ and $d_0,\dots,d_{2n}\in {\rm H}_3(M)$,
we call J\'onsson term of length $n$ in this paper,
that satisfy the following equations:
\begin{align*}
d_0=\pi_{(3,0)},d_{2n}=\pi_{(3,2)}\\
d_{i-1}(\pi_{(3,0)},\pi_{(3,1)},\pi_{(3,0)})=d_{i}(\pi_{(3,0)},\pi_{(3,1)},\pi_{(3,0)})
 &\ (1\leq i\leq 2n)\\
d_{2i}(\pi_{(3,0)},\pi_{(3,0)},\pi_{(3,2)})=d_{2i+1}(\pi_{(3,0)},\pi_{(3,0)},\pi_{(3,2)})
 &\ (0\leq i\leq n-1)\\
d_{2i-1}(\pi_{(3,0)},\pi_{(3,2)},\pi_{(3,2)})=d_{2i}(\pi_{(3,0)},\pi_{(3,2)},\pi_{(3,2)})
 &\ (1\leq i\leq n)
\end{align*}
\item
There exists an integer $n\geq 0$ such that, for each $s\in S$, there exist
$d_{s,i}\in M_{(3\times S,p_2,s)}$ ($0\leq i\leq 2n$) that satisfy the following equations:
\begin{align*}
&d_{s,0}=\pi_{(3\times S,p_2,(0,s))},d_{s,2n}=\pi_{(3\times S,p_2,(2,s))},&\\
&d_{s,i-1}((\pi_{(3\times S,p_2,(0,t))})_{t\in S},(\pi_{(3\times S,p_2,(1,t))})_{t\in S},(\pi_{(3\times S,p_2,(0,t))})_{t\in S})&\\
 =&d_{s,i}((\pi_{(3\times S,p_2,(0,t))})_{t\in S},(\pi_{(3\times S,p_2,(1,t))})_{t\in S},(\pi_{(3\times S,p_2,(0,t))})_{t\in S}),
 &\ (1\leq i\leq 2n)\\
&d_{s,2i}((\pi_{(3\times S,p_2,(0,t))})_{t\in S},(\pi_{(3\times S,p_2,(0,t))})_{t\in S},(\pi_{(3\times S,p_2,(2,t))})_{t\in S})&\\
 =&d_{s,2i+1}((\pi_{(3\times S,p_2,(0,t))})_{t\in S},(\pi_{(3\times S,p_2,(0,t))})_{t\in S},(\pi_{(3\times S,p_2,(2,t))})_{t\in S}),
 &\ (0\leq i\leq n-1)\\
&d_{s,2i-1}((\pi_{(3\times S,p_2,(0,t))})_{t\in S},(\pi_{(3\times S,p_2,(2,t))})_{t\in S},(\pi_{(3\times S,p_2,(2,t))})_{t\in S})&\\
 =&d_{s,2i}((\pi_{(3\times S,p_2,(0,t))})_{t\in S},(\pi_{(3\times S,p_2,(2,t))})_{t\in S},(\pi_{(3\times S,p_2,(2,t))})_{t\in S}),
 &\ (1\leq i\leq n)
\end{align*}
\item
For each $s\in S$, there exist an integer $n\geq 0$ and
$d_{s,i}\in M_{(3\times S,p_2,s)}$ ($0\leq i\leq 2n$) that satisfy
the same equations displayed in the previous condition.
\item
For each $s\in S$, there exist an integer $n\geq 0$ and
$d_{s,i}\in M_{3s\ra s}$ ($0\leq i\leq 2n$) that satisfy the following equations:
\[
\begin{array}{cl}
d_{s,0}=\pi_{(3s\ra s,0)},d_{s,2n}=\pi_{(3s\ra s,2)},&\\
d_{s,i-1}(\pi_{(3s\ra s,0)},\pi_{(3s\ra s,1)},\pi_{(3s\ra s,0)})
 =d_{s,i}(\pi_{(3s\ra s,0)},\pi_{(3s\ra s,1)},\pi_{(3s\ra s,0)}),
 &(1\leq i\leq 2n)\\
d_{s,2i}(\pi_{(3s\ra s,0)},\pi_{(3s\ra s,0)},\pi_{(3s\ra s,2)})
 =d_{s,2i+1}(\pi_{(3s\ra s,0)},\pi_{(3s\ra s,0)},\pi_{(3s\ra s,2)}),
 &(0\leq i\leq n-1)\\
d_{s,2i-1}(\pi_{(3s\ra s,0)},\pi_{(3s\ra s,2)},\pi_{(3s\ra s,2)})
 =d_{s,2i}(\pi_{(3s\ra s,0)},\pi_{(3s\ra s,2)},\pi_{(3s\ra s,2)}).
 &(1\leq i\leq n)
\end{array}
\]
\end{enumerate}
\end{thm}
\begin{proof}
Equivalence of 1,2,3 and 4 is similar to corresponding equivalence of Theorem \ref{malcev}. 2 $\Leftrightarrow$ 3 is known as J\'onsson's theorem (See e.g. {\cite[Theorem 12.6]{BS}}).
5 $\Rightarrow$ 4 easily follows from finiteness of $S$.
6 $\Rightarrow$ 5 is trivial.

1 $\Rightarrow$ 6. Let $F$ be a rank $(3\delta_{st})_{t\in S}$ free algebra of $\V(M)$,
$\{a,b,c\}$ be a set of free generators of $F$, and
$\theta,\eta,\zeta$ be the congruences of $F$ generated by $(a,c),(a,b),(b,c)$ respectively.
Then
\[
(a,c)\in \theta\land(\eta\lor\zeta)=(\theta\land\eta)\lor(\theta\land\zeta).
\]
By finiteness of arity of $M$, the join in congruence lattice is described as
$\alpha\lor\beta=\bigcup_{n\in\N}(\alpha\circ\beta)^n$, where $(\alpha\circ\beta)^n$
denotes the $k$-th relational power of $\alpha\circ\beta$.
Thus, there are $x_0,\dots,x_{2n}$ such that
\[
x_0=a,x_{2n}=c,(x_{2i},x_{2i+1})\in\theta\land\eta, (x_{2i-1},x_{2i})\in\theta\land\zeta.
\]
The terms $d_{s,i}$ corresponding to $x_i$ satisfy the equations stated in Condition 6.
\end{proof}

Next example is a counter example of infinitary variety for J\'onsson's Theorem.
\begin{eg}
Let $S=\N=\aleph_0$ and $C_n$ be the $<\mc\aleph_1$-ary clone defined by the following presentation:
\begin{itemize}
\item
The set of generators consists of 3-ary elements $d_0,\dots,d_{2n}$ and a nullary element $u$.
\item
The set of fundamental relations consists of J\'onsson equations, that is,
equations displayed in Condition 3 of the previous theorem.
\end{itemize}
Let $C$ be the direct product clone $\prod_{n\in\N}C_n$ and define
\[
d=(\pi_{(\aleph_0,i)})_{i\in\aleph_0}\in C_{\aleph_0},
e_n=(\pi_{(\{0\},0)},(u)_{i\in\aleph_0\sm\{n\}})\in C_1
\]
for $n\in\aleph_0$.
(Intuitively, $d$ and $e_n$ are the operations such that
$d:((a_{ij})_{j\in \N})_{i\in \N}\mapsto (a_{jj})_{j\in \N}$ and 
\mbox{$e_n:(a_i)_{i\in\N}\mapsto (a_n,(u)_{i\in\N\sm\{n\}})$}.) 
Then $(d,(e_n)_{n\in\N})$ is a diagonal pair of $C$.
Moreover, the following assertions hold.
\begin{enumerate}
\item
$\V(C)$ and $\V(C_d)$ are congruence distributive.
\item
$C$ does not have J\'onsson term.
\item
$C_d$ has terms that satisfy the equations stated in 6 of the previous theorem.
\item
$C_d$ is essentially finitary.
\end{enumerate}
\end{eg}
\begin{proof}
1. Each $C_n$ has J\'onsson term, thus $\V(C_n)$ is congruence distributive.
Thus,
\[
\V(C)=\V\left(\prod_{n\in\N}C_n\right)\iso \bigotimes_{n\in\N}\V(C_n)
\]
is congruence distributive. Here, $\otimes_{n\in\N}\V(C_n)$ denotes
non-indexed product of the family $\{\V(C_n)\}_{n\in \N}$ of varieties.
By Proposition \ref{sub-con-corresopndence}, $\V(C_d)$ also be congruence distributive.

2. If $C$ has J\'onsson term of length $n$, then $C_{n+1}$ also has J\'onsson term of length $n$.
It is verified by the induction on $n$ that this is impossible.

3. It is easily verified that $e_n(C)$ is isomorphic to $C_n$.
Particularly, $e_n(C)$ has J\'onsson term.
These are nothing but the terms of $C_d$ satisfying Condition 6 of the theorem.

4. Let $f\in C_{d,(\aleph_0,v,n)}$, where $v:\aleph_0\ra \N$ and $n\in\N$.
By the definition of heterogenization, $f$ is represented by $\tilde{f}\in C_{\aleph_0}$
such that $e_n\circ\tilde{f}=\tilde{f}$. This condition implies that
the $n$-th component of $\tilde{f}$ is an element of $C_n$ (it is finitary),
and other components are $u$ (depend no variables).
Therefore, $\tilde{f}$, particularly $f$, is finitary.
\end{proof}

\end{document}